\documentclass[12pt]{article}
\usepackage{amsmath,amsthm,amssymb}
\usepackage{enumitem}
\usepackage{xcolor}
\usepackage{adjustbox}

\usepackage{ marvosym }
\usepackage{ stmaryrd }

\usepackage{tikz}
 	 \usetikzlibrary{arrows,backgrounds}

\usepackage[pdftex,plainpages=false,hypertexnames=false,pdfpagelabels]{hyperref}
\newcommand{\arxiv}[1]{\href{http://arxiv.org/abs/#1}{\tt arXiv:\nolinkurl{#1}}}
\newcommand{\arXiv}[1]{\href{http://arxiv.org/abs/#1}{\tt arXiv:\nolinkurl{#1}}}
\newcommand{\doi}[1]{\href{http://dx.doi.org/#1}{{\tt DOI:#1}}}

\newcommand{\googlebooks}[1]{(preview at \href{http://books.google.com/books?id=#1}{google books})}

\usepackage{xcolor}
\definecolor{dark-red}{rgb}{0.7,0.25,0.25}
\definecolor{dark-blue}{rgb}{0.15,0.15,0.55}
\definecolor{medium-blue}{rgb}{0,0,.8}
\definecolor{DarkGreen}{RGB}{0,150,0}
\definecolor{rho}{named}{red}
\hypersetup{
   colorlinks, linkcolor={purple},
   citecolor={medium-blue}, urlcolor={medium-blue}
}

\usepackage{fullpage}

\theoremstyle{plain}
\newtheorem{thm}{Theorem}[section]
\newtheorem*{thm*}{Theorem}
\newtheorem{thmalpha}{Theorem}

\newtheorem{cor}[thm]{Corollary}

\newtheorem*{cor*}{Corollary}

\newtheorem*{conj*}{Conjecture}
\newtheorem{lem}[thm]{Lemma}
\newtheorem{prop}[thm]{Proposition}

\newtheorem*{quest*}{Question}
\newtheorem*{claim*}{Claim}

\theoremstyle{definition}

\newtheorem{defn}[thm]{Definition}

\newtheorem*{exs*}{Examples}
\newtheorem{ex}[thm]{Example}

\newtheorem{term}[thm]{Terminology}
\newtheorem{rem}[thm]{Remark}
\newtheorem*{siderem*}{Side remark}

\newtheorem*{warn*}{Warning}

\DeclareMathOperator{\End}{End}

\DeclareMathOperator{\Hom}{Hom}

\DeclareMathOperator{\op}{op}

\DeclareMathOperator{\id}{id}

\newcommand{\nc}[2]{\newcommand{#1}{#2}}
\nc{\B}{\mathrm{B}}
\nc{\C}{\mathbb{C}}
\nc{\R}{\mathbb{R}}
\nc{\Q}{\mathbb{Q}}
\nc{\Z}{\mathbb{Z}}
\nc{\N}{\mathbb{N}}
\nc{\cA}{\mathcal{A}}
\nc{\cB}{\mathcal{B}}
\nc{\cC}{\mathcal{C}}
\nc{\cD}{\mathcal{D}}
\nc{\bbD}{\mathbb{D}}
\nc{\cI}{\mathcal{I}}
\nc{\cM}{\mathcal{M}}
\nc{\g}{\mathfrak{g}}

\def\tworarrow{
    \hspace{.1cm}{
        \setlength{\unitlength}{.50mm}\linethickness{.09mm}
        \begin{picture}
            (8,8)(0,0)\qbezier(0,4)(4,7)(8,4)\qbezier(0,1)(4,-2)(8,1)\qbezier(3.5,4)(3.5,3)(3.5,1.5)
            \qbezier(4.5,4)(4.5,3)(4.5,1.5)\qbezier(4,0.8)(4.5,1.7)(5.5,2)\qbezier(4,0.8)(3.5,1.7)(2.5,2)
            \qbezier(8,1)(7.4,.2)(7.7,-.7)\qbezier(8,1)(7,1)(6.5,1.5)\qbezier(8,4)(7.4,4.8)(7.7,5.7)
            \qbezier(8,4)(7,4)(6.5,3.5)
        \end{picture}
        \hspace{.1cm}
    }
}

\newcommand{\Bim}{\mathrm{Bim}}
\newcommand\Hilb{\mathrm{Hilb}}
\newcommand{\Mod}{\mathrm{Mod}}
\newcommand{\Rep}{\mathrm{Rep}}

\def\semicolon{;}
\def\applytolist#1{
    \expandafter\def\csname multi#1\endcsname##1{
        \def\multiack{##1}\ifx\multiack\semicolon
            \def\next{\relax}
        \else
            \csname #1\endcsname{##1}
            \def\next{\csname multi#1\endcsname}
        \fi
        \next}
    \csname multi#1\endcsname}

\def\calc#1{\expandafter\def\csname c#1\endcsname{{\mathcal #1}}}
\applytolist{calc}QWERTYUIOPLKJHGFDSAZXCVBNM;
\def\bbc#1{\expandafter\def\csname bb#1\endcsname{{\mathbb #1}}}
\applytolist{bbc}QWERTYUIOPLKJHGFDSAZXCVBNM;
\def\bfc#1{\expandafter\def\csname bf#1\endcsname{{\mathbf #1}}}
\applytolist{bfc}QWERTYUIOPLKJHGFDSAZXCVBNM;
\def\sfc#1{\expandafter\def\csname s#1\endcsname{{\sf #1}}}
\applytolist{sfc}QWERTYUIOPLKJHGFDSAZXCVBNM;
\def\rmc#1{\expandafter\def\csname rm#1\endcsname{{\rm #1}}}
\applytolist{rmc}QWERTYUIOPLKJHGFDSAZXCVBNM;

\usetikzlibrary{shapes}
\usetikzlibrary{cd}
\usetikzlibrary{backgrounds}
\usetikzlibrary{decorations,decorations.pathreplacing,decorations.markings}
\usetikzlibrary{fit,calc,through}
\usetikzlibrary{external}
\tikzset{
	super thick/.style={line width=3pt}
}
\tikzstyle{shaded}=[fill=red!10!blue!20!gray!30!white]
\tikzstyle{unshaded}=[fill=white]
\tikzstyle{empty box}=[circle, draw, thick, fill=white, opaque, inner sep=2mm]
\tikzstyle{annular}=[scale=.7, inner sep=1mm, baseline]
\tikzstyle{rectangular}=[scale=.75, inner sep=1mm, baseline=-.1cm]
\tikzstyle{mid>}=[decoration={markings, mark=at position 0.5 with {\arrow{>}}}, postaction={decorate}]
\tikzstyle{mid<}=[decoration={markings, mark=at position 0.5 with {\arrow{<}}}, postaction={decorate}]
\tikzstyle{over}=[double, draw=white, super thick, double=]
\tikzstyle{knot}=[preaction={super thick, white, draw}]

\begin{document}
\title{Complete W*-categories}
\author{Andr\'e Henriques, Nivedita, and David Penneys}
\date{}
\maketitle
\begin{abstract}
We study $\rmW^*$-categories, and explain the ways in which complete $\rmW^*$-categories behave like categorified Hilbert spaces. Every $\rmW^*$-category $C$ admits a canonical categorified inner product
$\langle\,\,,\,\rangle_\Hilb\,:\,\overline C\times C\,\to\, \Hilb$.
Moreover, if $C$ and $D$ are complete $\rmW^*$-categories there is an antilinear equivalence
$$\dagger:\mathrm{Func}(C,D) \leftrightarrow \mathrm{Func}(D,C)$$
characterised by 
$\langle c,F^\dagger(d)\rangle_\Hilb
\simeq
\langle F(c),d\rangle_\Hilb$, for $c\in C$ and $d \in D$.
\end{abstract}

\maketitle


\section{Introduction}

$\rmW^*$-categories are the `many object' versions of $\rmW^*$-algebras (a.k.a.~von Neumann algebras).
They were introduced by Ghez, Lima and Roberts in \cite{MR808930}, but have not been the subject of much in depth investigation since then.

Our main observation is that complete\footnote{We use the terms \emph{complete} and \emph{Cauchy complete} as synonyms.} $\rmW^*$-categories behave in many ways like Hilbert spaces. Most notably, every $\rmW^*$-category $C$ admits a canonical sesquilinear functor
\[
\langle\,\,,\,\rangle_\Hilb:\overline C\times C\to \Hilb,
\]
first studied in \cite{MR2325696}, which we rename the `$\Hilb$-valued inner product'
(Definition~\ref{def: Hilb valued inner product}).
Given complete $\rmW^*$-categories $C$ and $D$, there is an antilinear equivalence of categories\vspace{-2mm}
\[
\dagger:\mathrm{Func}(C,D)\stackrel{\cong}{\longrightarrow} \mathrm{Func}(D,C)
\]
called \emph{adjoint},
characterised by the existence of unitary natural isomorphisms
\[
\langle c,F^\dagger(d)\rangle_{\Hilb}
\simeq
\langle F(c),d\rangle_{\Hilb}
\]
(Definition~\ref{def: adjoint functors}).\footnote{This is genuinely distinct from the usual notion of adjoint functors. \label{footnote 2}}
In the table on page~\pageref{pageref Table of analogies}, we present an extensive list of analogies between Hilbert spaces and complete $\rmW^*$-categories.

We start our paper with a quick review of some background material on Hilbert spaces and von Neumann algebras. We then define and study various notions of completeness for $\rmW^*$-categories. 
One of our main results is:
\begin{thmalpha}
There is an equivalence of bi-involutive bicategories
$\mathrm{vN2}\stackrel{\simeq}\to\mathrm{W^*Cat}$
between the bicategory of von Neumann algebras, bimodules, and intertwiners, and the bicategory of complete $\rmW^*$-categories admitting a set of generators, functors, and natural transformations.
\end{thmalpha}
The above result was essentially well known (\cite[Thm~2.3]{MR2325696}\footnote{Note that \cite[Thm~2.3]{MR2325696} was incorrectly stated, as it did not require the $\rmW^*$-categories to be complete.}); our contribution is to upgrade its statement from an equivalence of bicategories to an equivalence of bi-involutive bicategories, i.e., incorporating adjoints.
We also present a variant in the context of small $\rmW^*$-categories:

\begin{thmalpha}
Let $\kappa$ be an infinite cardinal. Then there is an equivalence of bi-involutive bicategories between the bicategories of:\\[1mm]
{\it (i)} $\kappa$-separable von Neumann algebras, $\kappa$-separable bimodules, and intertwiners.
\\[1mm]
{\it (ii)} $\kappa$-small $\kappa$-complete
$\rmW^*$-categories, functors, and natural transformations.
\\[1mm]
{\it (iii)} $\kappa$-generated complete $\rmW^*$-categories,
$\kappa$-small functors, and natural transformations.
\end{thmalpha}

In the table below, we present an extensive list of analogies between the notion of Hilbert space and that of complete $\rmW^*$-category (all the $\rmW^*$-categories in the table are assumed to admit a set of generators).
Later, in the appendix, we reorganise some of those elements into a more compact diagram. 
The diagram also incorporates bicommutant categories, first introduced in \cite{MR3747830}, as the categorical analogs of von Neumann algebras, as well as speculations about how the story might extend to higher categories.

\paragraph{\large Table of analogies} {\phantom :}\label{pageref Table of analogies}
\\

\def\wi{8cm}

\noindent $ $ \hspace{-5mm}
\begin{tikzpicture}
\draw (-8.25,.9) -- (8.25,.9);
\draw (-8.25,0) --node[above, yshift=6]{\sc Table of analogies between Hilbert spaces and complete $\rmW^*$-categories
} (8.25,0);
\node [matrix, below] (my matrix) at (0,0)
{\node {\parbox{\wi}{
A Hilbert space $H$.
}}; &[.3cm] \node  {\parbox{\wi}{
A complete $\rmW^*$-category $C$ (Definition~\ref{def: Cauchy completion}).
}};
\\ };
\draw (0,0) -- ($(my matrix.south)+(0,.07)$);
\end{tikzpicture}
\vspace{-6.2mm}

\noindent $ $ \hspace{-5mm}
\begin{tikzpicture}
\draw (-8.25,0) -- (8.25,0);
\node [matrix, below] (my matrix) at (0,0)
{\node {\parbox{\wi}{
The one dimensional Hilbert space $\bbC$.
}}; &[.3cm] \node  {\parbox{\wi}{
The $\rmW^*$-category $\Hilb$.
}};
\\ };
\draw (0,0) -- ($(my matrix.south)+(0,.07)$);
\end{tikzpicture}
\vspace{-6.2mm}

\noindent $ $ \hspace{-5mm}
\begin{tikzpicture}
\draw (-8.25,0) -- (8.25,0);
\node [matrix, below] (my matrix) at (0,0)
{\node {\parbox{\wi}{
A complex number.
}}; &[.3cm] \node  {\parbox{\wi}{
A Hilbert space.
}};
\\ };
\draw (0,0) -- ($(my matrix.south)+(0,.07)$);
\end{tikzpicture}
\vspace{-6.2mm}

\noindent $ $ \hspace{-5mm}
\begin{tikzpicture}
\draw (-8.25,0) -- (8.25,0);
\node [matrix, below] (my matrix) at (0,0)
{\node {\parbox{\wi}{
A pre-Hilbert space.
}}; &[.3cm] \node  {\parbox{\wi}{
A $\rmW^*$-category with direct sums, but potentially not idempotent complete.
}};
\\ };
\draw (0,0) -- ($(my matrix.south)+(0,.07)$);
\end{tikzpicture}
\vspace{-6.2mm}

\noindent $ $ \hspace{-5mm}
\begin{tikzpicture}
\draw (-8.25,0) -- (8.25,0);
\node [matrix, below] (my matrix) at (0,0)
{\node {\parbox{\wi}{
Addition $+:H \times H \to H$.
}}; &[.3cm] \node  {\parbox{\wi}{
Orthogonal direct sum $\oplus:C\times C\to C$.
}};
\\ };
\draw (0,0) -- ($(my matrix.south)+(0,.07)$);
\end{tikzpicture}
\vspace{-6.2mm}

\noindent $ $ \hspace{-5mm}
\begin{tikzpicture}
\draw (-8.25,0) -- (8.25,0);
\node [matrix, below] (my matrix) at (0,0)
{\node {\parbox{\wi}{
Scalar multiplication $\cdot:\bbC \times H \to H$.
}}; &[.3cm] \node  {\parbox{\wi}{
The canonical tensoring $\otimes:\Hilb\times C\to C$
(Definition~\eqref{eq: action of Hilb}).
}};
\\ };
\draw (0,0) -- ($(my matrix.south)+(0,.07)$);
\end{tikzpicture}
\vspace{-6.2mm}

\noindent $ $ \hspace{-5mm}
\begin{tikzpicture}
\draw (-8.25,0) -- (8.25,0);
\node [matrix, below] (my matrix) at (0,0)
{\node {\parbox{\wi}{
The inner product\\[2mm] \centerline{$\langle\,\,,\,\rangle:\overline H\times H\to \bbC$}\\[2mm] on a (pre-)Hilbert space $H$.
}}; &[.3cm] \node  {\parbox{\wi}{
The canonical $\Hilb$-valued inner product \\[2mm] \centerline{$\langle\,\,,\,\rangle_\Hilb:\overline C\times C\to \Hilb$}\\[2mm]
on a $\rmW^*$-category (Definition~\ref{def: Hilb valued inner product}).
}};
\\ };
\draw (0,0) -- ($(my matrix.south)+(0,.07)$);
\end{tikzpicture}
\vspace{-6.2mm}

\noindent $ $ \hspace{-5mm}
\begin{tikzpicture}
\draw (-8.25,0) -- (8.25,0);
\node [matrix, below] (my matrix) at (0,0)
{\node {\parbox{\wi}{
A pre-Hilbert space $H$ is complete iff\\[3mm]
\centerline{$H \to \{f: \overline H\to \bbC \,|\, \text{$f$ is bounded}\}$}\\
${}$ $\quad\qquad\xi \,\mapsto\, \langle -,\xi\rangle$\\[2mm]
is an isomorphism (Lemma~\ref{lem: pre-Hilbert space}).
}}; &[.3cm] \node  {\parbox{\wi}{
A $\rmW^*$-category is complete iff
\\[3mm]
\centerline{$C \to \mathrm{Func}(\overline {C},\Hilb)$}\\
\centerline{$\!\!c \mapsto\, \langle-,c\rangle_{\Hilb}$}\\[2mm]
is an equivalence (Proposition~\ref{prop:riesz}).
}};
\\ };
\draw (0,0) -- ($(my matrix.south)+(0,.07)$);
\end{tikzpicture}
\vspace{-6.2mm}

\noindent $ $ \hspace{-5mm}
\begin{tikzpicture}
\draw (-8.25,0) -- (8.25,0);
\node [matrix, below] (my matrix) at (0,0)
{\node {\parbox{\wi}{
A positive number $x\in\bbR_{\ge 0}$.
}}; &[.3cm] \node  {\parbox{\wi}{
A von Neumann algebra $A$.
}};
\\ };
\draw (0,0) -- ($(my matrix.south)+(0,.07)$);
\end{tikzpicture}
\vspace{-6.2mm}

\noindent $ $ \hspace{-5mm}
\begin{tikzpicture}
\draw (-8.25,0) -- (8.25,0);
\node [matrix, below] (my matrix) at (0,0)
{\node {\parbox{\wi}{
The square-norm $\|\xi\|^2$ of $\xi\in H$.
}}; &[.3cm] \node  {\parbox{\wi}{
The endomorphism algebra $\End(c)$ of $c\in C$.
}};
\\ };
\draw (0,0) -- ($(my matrix.south)+(0,.07)$);
\end{tikzpicture}
\vspace{-6.2mm}

\noindent $ $ \hspace{-5mm}
\begin{tikzpicture}
\draw (-8.25,0) -- (8.25,0);
\node [matrix, below] (my matrix) at (0,0)
{\node {\parbox{\wi}{
The map $\bbC\to \bbR_{\ge0}:z\mapsto |z|^2$.
}}; &[.3cm] \node  {\parbox{\wi}{
The construction $\Hilb\to\operatorname{vNA}:H\mapsto B(H)$.
}};
\\ };
\draw (0,0) -- ($(my matrix.south)+(0,.07)$);
\end{tikzpicture}
\vspace{-6.2mm}

\noindent $ $ \hspace{-5mm}
\begin{tikzpicture}
\draw (-8.25,0) -- (8.25,0);
\node [matrix, below] (my matrix) at (0,0)
{\node {\parbox{\wi}{
The inclusion $\bbR_{\ge0}\hookrightarrow \bbC$.
}}; &[.3cm] \node  {\parbox{\wi}{
The construction $\operatorname{vNA}\to\Hilb:A\mapsto L^2A$. 
}};
\\ };
\draw (0,0) -- ($(my matrix.south)+(0,.07)$);
\end{tikzpicture}
\vspace{-6.2mm}

\noindent $ $ \hspace{-5mm}
\begin{tikzpicture}
\draw (-8.25,0) -- (8.25,0);
\node [matrix, below] (my matrix) at (0,0)
{\node {\parbox{\wi}{
The equation\\
\centerline{$\langle\xi,\xi\rangle=\|\xi\|^2$.}
}}; &[.3cm] \node  {\parbox{\wi}{
For $c\in C$, 
there is a canonical unitary \\[2mm]
\centerline{$\langle c,c\rangle_{\Hilb}\cong L^2(\End(c))$.} 
}};
\\ };
\draw (0,0) -- ($(my matrix.south)+(0,.07)$);
\end{tikzpicture}
\vspace{-6.2mm}

\noindent $ $ \hspace{-5mm}
\begin{tikzpicture}
\draw (-8.25,0) -- (8.25,0);
\node [matrix, below] (my matrix) at (0,0)
{\node {\parbox{\wi}{
A set of vectors $e_i\in H$ which are orthogonal, and span a dense subspace.
}}; &[.3cm] \node  {\parbox{\wi}{
An orthogonal set of generators of a $\rmW^*$-category
(Definitions~\ref{def: admit a set of generators} and~\ref{def: orthog gens}).
}};
\\ };
\draw (0,0) -- ($(my matrix.south)+(0,.07)$);
\end{tikzpicture}
\vspace{-6.2mm}

\noindent $ $ \hspace{-5mm}
\begin{tikzpicture}
\draw (-8.25,0) -- (8.25,0);
\node [matrix, below] (my matrix) at (0,0)
{\node {\parbox{\wi}{
Isometry between (pre-)Hilbert spaces.
}}; &[.3cm] \node  {\parbox{\wi}{
Fully faithful functor between $\rmW^*$-cate\-go\-ries.
}};
\\ };
\draw (0,0) -- ($(my matrix.south)+(0,.07)$);
\end{tikzpicture}
\vspace{-6.2mm}

\noindent $ $ \hspace{-5mm}
\begin{tikzpicture}
\draw (-8.25,0) -- (8.25,0);
\node [matrix, below] (my matrix) at (0,0)
{\node {\parbox{\wi}{
$\|f\|^2$, the square of the operator norm of a bounded linear map $f:H\to K$ between (pre-)Hilbert spaces.
}}; &[.3cm] \node  {\parbox{\wi}{
The endomorphism algebra $\End(F)$ of a functor $F$ between $\rmW^*$-categories.
}};
\\ };
\draw (0,0) -- ($(my matrix.south)+(0,.07)$);
\end{tikzpicture}
\vspace{-6.2mm}

\noindent $ $ \hspace{-5mm}
\begin{tikzpicture}
\draw (-8.25,0) -- (8.25,0);
\node [matrix, below] (my matrix) at (0,0)
{\node {\parbox{\wi}{
The equation\,
$\|f\|^2\|g\|^2 \ge \|f\circ g\|^2$.
}}; &[.3cm] \node  {\parbox{\wi}{
There is a canonical bilinear map\\ \centerline{$\End(F)\times \End(G)\to \End(F\circ G)$.}
}};
\\ };
\draw (0,0) -- ($(my matrix.south)+(0,.07)$);
\end{tikzpicture}
\vspace{-6.2mm}

\noindent $ $ \hspace{-5mm}
\begin{tikzpicture}
\draw (-8.25,0) -- (8.25,0);
\node [matrix, below] (my matrix) at (0,0)
{\node {\parbox{\wi}{
The Hilbert space $\bbC \Omega_x$ freely generated by a vector $\Omega_x$ of square-norm $x\in\bbR_{\ge0}$ (it is one-dimensional if $x>0$, and zero if $x=0$).
}}; &[.3cm] \node  {\parbox{\wi}{
Given a von Neumann algebra $A$, the category $A^{\op}$-$\Mod$ has a distinguished object $L^2A^{\op}$, whose endomorphism algebra is~$A$.
}};
\\ };
\draw (0,0) -- ($(my matrix.south)+(0,.07)$);
\end{tikzpicture}
\vspace{-6.2mm}

\noindent $ $ \hspace{-5mm}
\begin{tikzpicture}
\draw (-8.25,0) -- (8.25,0);
\node [matrix, below] (my matrix) at (0,0)
{\node {\parbox{\wi}{
Given a vector $\xi\in H$ in a Hilbert space, and a number $x\ge \|\xi\|^2$, the map $\Omega_x\mapsto \xi$ extends uniquely to a contracting linear map\\[2mm]
\centerline{$\bbC \Omega_x\to H$}\\
\centerline{$a \Omega_x\mapsto a\xi$.\hspace{-2mm}}
}}; &[.3cm] \node  {\parbox{\wi}{
For $c\in C$ in a complete $\rmW^*$-category, and $A\to \End(c)$ a homomorphism, the map $L^2A^{\op}\mapsto c$ extends uniquely to a functor \\[2mm]
\centerline{
$A^{\op}$-$\Mod \to C$
:
$H_A\mapsto H\boxtimes_A c$}\\[3mm]
(Corollary~\ref{cor: enough to construct on generator} and Definition~\eqref{eq: Connes fusion with an object}).
}};
\\ };
\draw (0,0) -- ($(my matrix.south)+(0,.07)$);
\end{tikzpicture}
\vspace{-6.2mm}

\noindent $ $ \hspace{-5mm}
\begin{tikzpicture}
\draw (-8.25,0) -- (8.25,0);
\node [matrix, below] (my matrix) at (0,0)
{\node {\parbox{\wi}{
An orthogonal spanning set $\{e_i\}$ for a Hilbert space $H$ yields a canonical unitary\\[3mm]
\centerline{$\displaystyle H\,\cong\,\bigoplus_i \bbC\Omega_{\|e_i\|^2}$.}
}}; &[.3cm] \node  {\parbox{\wi}{
An orthogonal set of generators $\{c_i\}$ for a complete $\rmW^*$-category $C$ yields a canonical equivalence
\\[3mm]
\centerline{$\displaystyle C\,\cong\, \boxplus_i\, \overline{\End(c_i)}\text{-}\Mod$.}\\[3mm]
(Lemma~\ref{lem: boxplus decomposition} and Proposition~\ref{prop:freydembedding'})
}};
\\ };
\draw (0,0) -- ($(my matrix.south)+(0,.07)$);
\end{tikzpicture}
\vspace{-6.2mm}

\noindent $ $ \hspace{-5mm}
\begin{tikzpicture}
\draw (-8.25,0) -- (8.25,0);
\node [matrix, below] (my matrix) at (0,0)
{\node {\parbox{\wi}{
An orthogonal spanning set $\{e_i\}$ of $H$ affords a unique expansion\\[3mm]
\centerline{$\displaystyle \xi \,=\, \sum_i\frac{\langle e_i,\xi\rangle e_i}{\|e_i\|^2}$
for every $\xi \in H$.}
}}; &[.3cm] \node  {\parbox{\wi}{
An orthogonal set of generators $\{c_i\}$ of $C$
affords a unique expression
\\[2mm]
\centerline{$\displaystyle x \,=\, \bigoplus_i\, \langle c_i,\underset{\End(c_i)}{x \rangle\,\,\boxtimes\,\,\, c_i}
$
for every $x\in C$.}
}};
\\ };
\draw (0,0) -- ($(my matrix.south)+(0,.07)$);
\end{tikzpicture}
\vspace{-6.2mm}

\noindent $ $ \hspace{-5mm}
\begin{tikzpicture}
\draw (-8.25,0) -- (8.25,0);
\node [matrix, below] (my matrix) at (0,0)
{\node {\parbox{\wi}{
Given Hilbert spaces $H$ and $K$, there is an antilinear isometry called `adjoint'\\[2.5mm]
\centerline{$*:\Hom(H,K)\stackrel{\cong}{\to} \Hom(K,H)$}
\\[2.5mm]
characterised by 
\\[2mm]
\centerline{$\displaystyle \langle \xi,f^*\eta\rangle
=\langle f\xi,\eta\rangle$.}
}}; &[.3cm] \node  {\parbox{\wi}{
Given complete $\rmW^*$-categories $C$ and $D$, there is an antilinear equivalence, `adjoint',
\\[2.5mm]
\centerline{$\displaystyle \dagger:\mathrm{Func}(C,D)\stackrel{\cong}{\to} \mathrm{Func}(D,C)$}\\[2.5mm]
characterised by unitary isomorphisms
\\[2.5mm]
\centerline{$\displaystyle \langle c,F^\dagger(d)\rangle
\simeq
\langle F(c),d\rangle$,}\\[2.5mm]
natural in $c$ and $d$
(Definition~\ref{def: adjoint functors}).
}};
\\ };
\draw (0,0) -- ($(my matrix.south)+(0,.07)$);
\end{tikzpicture}
\vspace{-6.2mm}

\noindent $ $ \hspace{-5mm}
\begin{tikzpicture}
\draw (-8.25,0) -- (8.25,0);
\node [matrix, below] (my matrix) at (0,0)
{\node {\parbox{\wi}{
Given two vectors $\xi,\eta\in H$,
the composite\\[2mm]
\centerline{$(\xi:\bbC\to H)^*\circ (\eta:\bbC\to H)$}
\\[2mm]
is multiplication by
$\langle \xi,\eta\rangle:\bbC\to \bbC$.
}}; &[.3cm] \node  {\parbox{\wi}{
Given two objects $c,d\in C$,
the composite\\[2mm]
\centerline{$(c:\Hilb\to C)^\dagger\circ (d:\Hilb\to C)$.}\\[2mm]
is tensoring with
$\langle c,d\rangle_{\Hilb}:\Hilb\to \Hilb$.
}};
\\ };
\draw (0,0) -- ($(my matrix.south)+(0,.07)$);
\end{tikzpicture}
\vspace{-6.2mm}

\noindent $ $ \hspace{-5mm}
\begin{tikzpicture}
\draw (-8.25,0) -- (8.25,0);
\node [matrix, below] (my matrix) at (0,0)
{\node {\parbox{\wi}{
A map $F:H\to K$ 
is completely determined by
$F(e_i)\in K$ via the formula
\\[3mm]
\centerline{$\displaystyle 
F(\xi) \,=\, \sum_i\frac{\langle e_i,\xi\rangle F(e_i)}{\|e_i\|^2}
$}
\\[2mm]
where $\{e_i\}$ is an orthogonal basis of $H$.
}}; &[.3cm] \node  {\parbox{\wi}{
A functor $F:C\to D$ is determined by $F(c_i)\in D$ along with the actions of $\End(c_i)$ on $F(c_i)$, by
\\[2.5mm]
\centerline{$\displaystyle F(x) \,=\, \bigoplus_i \underset{\End(c_i)}{\langle c_i,x \rangle\,\,\boxtimes\,\,\, F(c_i).}
$}
\\[1mm]
where $\{c_i\}$ is an orthogonal set of generators of $C$
(Remark~\ref{rem: witty}).
}};
\\ };
\draw (0,0) -- ($(my matrix.south)+(0,.07)$);
\end{tikzpicture}
\vspace{-6.2mm}

\noindent $ $ \hspace{-5mm}
\begin{tikzpicture}[ampersand replacement=\&]
\draw (-8.25,0) -- (8.25,0);
\node [matrix, below] (my matrix) at (0,0)
{\node {\parbox{\wi}{
A map $F:H\to K$ is determined by 
\\[2mm]
\centerline{$\displaystyle
F(\xi) \,=\, \sum_{i,j} \frac{\,\,\langle e_i,\xi\rangle\, a_{ij}\, f_j\,\,}{\|e_i\|^2\|f_j\|^2}$}
\\[2mm]
where $a_{ij}=\langle f_j,F(e_i)\rangle$,
and
$\{e_i\}$, $\{f_j\}$ are orthogonal bases of $H$ and $K$.
}}; \&[.3cm] \node  {\parbox{\wi}{
A functor $F:C\to D$ between complete $\rmW^*$-categories
is determined by
\\[2mm]
\centerline{
$\displaystyle
F(x) \,=\, \bigoplus_{i,j}\, \langle c_i,x\rangle\underset{\End(c_i)}{\boxtimes} X_{ij}\underset{\End(d_j)}{\boxtimes} d_j$}
\\[2mm]
where $X_{ij}=\langle d_j,F(c_i)\rangle$, and
$\{c_i\}$, $\{d_j\}$ are orthogonal generating sets for $C$ and $D$.
}};
\\ };
\draw (0,0) -- ($(my matrix.south)+(0,.07)$);
\draw ($(my matrix.south)+(0,.07)+(8.25,0)$) -- ($(my matrix.south)+(0,.07)+(-8.25,0)$);
\end{tikzpicture}

\vspace{1cm}

\paragraph{\large Summary of notations and definitions:}
\begin{itemize}
\item[-]
$B(H)\;$ Bounded operators on a Hilbert space
\hfill page~\pageref{pageref B(H)}

\vspace{-3mm}\item[-]
von Neumann algebras \hfill
Definition~\ref{def: vNalg}, page~\pageref{def: vNalg}

\vspace{-3mm}\item[-]
$\bar\otimes\;$ The spatial tensor product of von Neumann algebras
\hfill page~\pageref{pageref spatial tensor product}

\vspace{-3mm}\item[-]
$\prod A_i\;$ The product of von Neumann algebras
\hfill page~\pageref{pageref product vNalg}

\vspace{-3mm}\item[-]
$L^2A\;$ The standard form of a von Neumann algebra
\hfill page~\pageref{pageref Standard Form}

\vspace{-3mm}\item[-]
$\boxtimes_A\;$ The Connes fusion tensor product \hfill Definition~\ref{def: Connes fusion}, page~\pageref{def: Connes fusion}

\vspace{-3mm}\item[-]
$\overline C\;$ The complex conjugate of a $\rmW^*$-category \hfill page~\pageref{pageref: complex conjugate}

\vspace{-3mm}\item[-]
$C^\oplus\;$ The completion of a $\rmW^*$-category under finite direct sums \hfill
page~\pageref{pageref: completion under finite direct sums}

\vspace{-3mm}\item[-]
$\rmW^*$-categories \hfill Definition~\ref{def: W-star cat},
page~\pageref{def: W-star cat}

\vspace{-3mm}\item[-]
$\Hilb\;$ The category of Hilbert spaces \hfill page~\pageref{category Hilb}

\vspace{-3mm}\item[-]
$A$-$\Mod\;$ The category of modules over a von Neumann algebra \hfill page~\pageref{category AMod}

\vspace{-3mm}\item[-]
$\mathbf{B}A\;$ The $\rmW^*$-category with exactly one object $\star_A$ \hfill page~\pageref{category BA}

\vspace{-3mm}\item[-]
Functors between $\rmW^*$-categories \hfill Definition~\ref{def: functors linear and bilinear},
page~\pageref{def: functors linear and bilinear}

\vspace{-3mm}\item[-]
Faithful and fully faithful functors
\hfill page~\pageref{pageref: (fully) faithful functor}

\vspace{-3mm}\item[-]
Dominant functors
\hfill page~\pageref{pageref: dominant functor}

\vspace{-3mm}\item[-]
Idempotent complete $\rmW^*$-categories  \hfill Definition~\ref{def: idempotent complete},
page~\pageref{def: idempotent complete}

\vspace{-3mm}\item[-]
$\hat C\;$ The idempotent completion of a $\rmW^*$-category  \hfill page~\pageref{page ref: idempotent completion}

\vspace{-3mm}\item[-]
Generator of a $\rmW^*$-category  \hfill
Definition~\ref{def: admit a set of generators},
page~\pageref{def: admit a set of generators}

\vspace{-3mm}\item[-]
To admit a (set of) generator(s) \hfill
Definition~\ref{def: admit a set of generators},
page~\pageref{def: admit a set of generators}

\vspace{-3mm}\item[-]
$\mathrm{Func}(C,D)\;$ The $\rmW^*$-category of functors between two $\rmW^*$-categories \hfill
Df~\ref{def: functor categories},
page~\pageref{def: functor categories}

\vspace{-3mm}\item[-]
$\mathrm{Bilin}(C_1\times C_2,D)\;$ The $\rmW^*$-category of bilinear functors between $\rmW^*$-categories \hfill
page~\pageref{pageref: biliar functors}

\vspace{-3mm}\item[-]
$\oplus c_i\;$ Orthogonal direct sum of objects in a $\rmW^*$-category \hfill Definition~\ref{def: orthogonal direct sum},
page~\pageref{def: orthogonal direct sum}

\vspace{-3mm}\item[-]
For a $\rmW^*$-category to admit all direct sums \hfill page~\pageref{pageref: to admit all direct sums}

\vspace{-3mm}\item[-]
$C^{\bar\oplus}\;$ 
The direct sum completion of a $\rmW^*$-category \hfill
page~\pageref{pageref: direct sum completion}

\vspace{-3mm}\item[-]
Cauchy complete $\rmW^*$-categories $\equiv$ complete $\rmW^*$-categories \hfill Definition~\ref{def: Cauchy completion}, page~\pageref{def: Cauchy completion}

\vspace{-3mm}\item[-]
$C^{\hat \oplus}\;$ The Cauchy completion of a $\rmW^*$-category \hfill Definition~\ref{def: Cauchy completion}, page~\pageref{def: Cauchy completion}

\vspace{-3mm}\item[-]
The full subcategory generated by a (set of) object(s)
\hfill Definition~\ref{def: full subcategory generated by}, page~\pageref{def: full subcategory generated by}

\vspace{-3mm}\item[-]
$c \cong_{st} d\;$ Two objects of a $\rmW^*$-category are stably equivalent \hfill Definition~\ref{def: stably equivalent objects}, page~\pageref{def: stably equivalent objects}

\vspace{-3mm}\item[-]
$\boxplus C_i\;$ The orthogonal direct sum of $\rmW^*$-categories \hfill Definition~\ref{def: orthogonal direct sum of WCat}, page~\pageref{def: orthogonal direct sum of WCat}

\vspace{-3mm}\item[-]
$\boxplus c_i\;$ An object of $\boxplus C_i$ \hfill Definition~\ref{def: orthogonal direct sum of WCat}, page~\pageref{def: orthogonal direct sum of WCat}

\vspace{-3mm}\item[-]
$\otimes\;$ The action of $\Hilb$ on a $\rmW^*$-category \hfill
page~\pageref{eq: action of Hilb}

\vspace{-3mm}\item[-]
$C\bar\otimes D\;$ The tensor product of $\rmW^*$-categories \hfill
page~\pageref{pageref tens prod of WCat}

\vspace{-3mm}\item[-]
$C \hat \otimes D\;$ Tensor product of 
complete $\rmW^*$-categories \hfill
Definition~\ref{def: completed tensor product}, page~\pageref{def: completed tensor product}

\vspace{-3mm}\item[-]
$\boxtimes_A\;$ Connes fusion with an $A$-module internal to a $\rmW^*$-category \hfill
page~\pageref{eq: Connes fusion with an object}

\vspace{-3mm}\item[-]
$\langle \,\,,\,\rangle_\Hilb$ The Hilb-valued inner product on a $\rmW^*$-category \hfill Definition~\ref{def: Hilb valued inner product},
page~\pageref{def: Hilb valued inner product}

\vspace{-3mm}\item[-]
Orthogonal generators for a $\rmW^*$-category
\hfill Definition~\ref{def: orthog gens},
page~\pageref{def: orthog gens}

\vspace{-3mm}\item[-]
$F^\dagger\;$ The adjoint of a functor between $\rmW^*$-categories
\hfill
Definition~\ref{def: adjoint functors}, page~\pageref{eq: dag}

\vspace{-3mm}\item[-]
$P^{\mathsf v}_F\subset \End(F)\;$
The vertical cone
\hfill page~\pageref{pageref vertical cone}

\vspace{-3mm}\item[-]
$P^{\mathsf h}_{F,G}\subset \Hom(F^\dag\circ F,G^\dag\circ G)\;$
The horizontal cone
\hfill page~\pageref{pageref horizontal cone}

\vspace{-3mm}\item[-]
$\kappa\;$ An infinite cardinal
 \hfill page~\pageref{sec: Small W* categories}

\vspace{-3mm}\item[-]
$\kappa$-separable Hilbert spaces, $\kappa$-separable von Neumann algebras \hfill
page~\pageref{pageref k-separable}

\vspace{-3mm}\item[-]
$\kappa$-small object of a $\rmW^*$-category \hfill
Definition~\ref{def: k-small}, page~\pageref{def: k-small}

\vspace{-3mm}\item[-]
Locally $\kappa$-small $\rmW^*$-category \hfill
Definition~\ref{def: (locally) kappa-small}, page~\pageref{def: (locally) kappa-small}

\vspace{-3mm}\item[-]
$\kappa$-small $\rmW^*$-category \hfill
Definition~\ref{def: (locally) kappa-small}, page~\pageref{def: (locally) kappa-small}

\vspace{-3mm}\item[-]
$C_{<\kappa}\;$ The subcategory of $\kappa$-small objects of a $\rmW^*$-category
\hfill page~\pageref{pageref: C_<k}

\vspace{-3mm}\item[-]
Locally $\kappa$-generated $\rmW^*$-category
\hfill
Definition~\ref{def: (locally) k-generated},
page~\pageref{def: (locally) k-generated}

\vspace{-3mm}\item[-]
$\kappa$-generated $\rmW^*$-category
\hfill
Definition~\ref{def: (locally) k-generated},
page~\pageref{def: (locally) k-generated}

\vspace{-3mm}\item[-]
Admitting all $\kappa$-small direct sums \hfill page~\pageref{pageref: to admit all k-small direct sums}

\vspace{-3mm}\item[-]
$\kappa$-
complete $\rmW^*$-categories 
($\equiv$ $\kappa$-Cauchy complete)
\hfill page~\pageref{pageref: k-Cauchy complete}

\vspace{-3mm}\item[-]
$C^{\hat \oplus_\kappa}\;$ The $\kappa$-
completion of a $\rmW^*$-category
\hfill Definition~\ref{def: k-Cauchy completion}, page~\pageref{def: k-Cauchy completion}

\vspace{-3mm}\item[-]
$\boxplus^{<\kappa} C_i\;$ A certain full subcategory of $\boxplus C_i\;$
\hfill
page~\pageref{pageref boxplus^<k}

\vspace{-3mm}\item[-]
Maximal object in a $\rmW^*$-category \hfill page~\pageref{pageref: maximal}

\vspace{-3mm}\item[-]
$\rmW^*$-tensor categories \hfill page~\pageref{pageref: W^*-tensor category}

\vspace{-3mm}\item[-]
Bi-involutive $\rmW^*$-tensor categories
\hfill Definition~\ref{def: bi-involutive tensor category}, page~\pageref{def: bi-involutive tensor category} 

\vspace{-3mm}\item[-]
$\mathrm{vN2}$ The bicategory of von Neumann algebras
\hfill page~\pageref{pageref: vN2}

\vspace{-3mm}\item[-]
$\mathrm{W^*Cat}$ The bicategory of $\rmW^*$-categories
\hfill page~\pageref{pageref: WstarCat}

\end{itemize}

\paragraph{Acknowledgements}

David Penneys was supported by NSF DMS-2154389
and also by NSF DMS-1928930 while he was in residence at the Mathematical Sciences Research Institute/SLMath in Berkeley, California, during Summer 2024.

For the purpose of Open Access, the authors have applied a CC BY public copyright licence to any Author Accepted Manuscript (AAM) version arising from this submission.
AH~would like to thank Yutong Dai, Lucas Hataishi, and Adri\`a Marin Salvador for proofreading this paper.



\section{Hilbert spaces and von Neumann algebras}

In this section, we recall without proof basic results and definitions about Hilbert spaces and von Neumann algebras. We assume general familiarity with the structure and theory of von Neumann algebras.
Recall that a pre-Hilbert space is a vector space equipped with a positive definite sesquilinear form.

\begin{lem}\label{lem: pre-Hilbert space}
A pre-Hilbert space $H$ is a Hilbert space (i.e. $H$ is complete) if and only if the map
\[
H \to \{f: \overline H\to \bbC \,|\, \text{$f$ is bounded}\}:
\xi \mapsto \langle -,\xi\rangle
\]
is an isomorphism. \hfill $\square$
\end{lem}

If $H$ is a Hilbert space,
we write $B(H)$ for the \emph{$*$-algebra of bounded operators}\label{pageref B(H)} on $H$.

\begin{defn}\label{def: vNalg}
A \emph{von Neumann algebra} on $H$ is a $*$-subalgebra of $B(H)$ such that $A''=A$, where $A':=\{b\in B(H):ab=ba,\forall a\in A\}$ denotes the commutant of $A$, and $A'':=(A')'$.
\end{defn}

Given two von Neumann algebras $A\subset B(H)$
and $B\subset B(K)$, their \emph{spatial tensor product}\label{pageref spatial tensor product} $A\bar\otimes B$ is the von Neumann algebra generated by the algebraic tensor product $A\otimes^{alg} B$ acting on the Hilbert space tensor product $H\otimes K$.

Given a collection of von Neumann algebras $A_i\subset B(H_i)$ indexed by some set $I$, their \emph{product}\label{pageref product vNalg} $\prod_{i\in I}A_i$
is the von Neumann algebra generated by the algebraic direct sum $\bigoplus^{alg}_{i\in I}A_i$ acting on the Hilbert space direct sum $\bigoplus_{i\in I}H_i$.
When the indexing set is finite, we also write $\bigoplus_{i\in I}A_i$ in place of $\prod_{i\in I}A_i$.
\medskip

A $*$-algebra homomorphism $f:A\to B$ between von Neumann algebras is called \emph{normal} if $f(\sup a_i)=\sup f(a_i)$ for every bounded increasing net $a_i$ of positive elements of $A$.
\begin{defn}
A \emph{(von Neumann algebra) homomorphism} between two von Neumann algebras is a normal unital $*$-algebra homomorphism.
\end{defn}

Homomorphisms between von Neumann algebras behave much more rigidly than homomorphisms between algebras. For example:

\begin{lem}\label{lem: kernel of maps of vNalg}
If $f:A\to B$ is a von Neumann algebra homomorphism, then $A$ splits as a von Neumann algebra direct sum $A=A_0\oplus A_1$, were $A_0=\ker(f)$ and $A_1\cong \mathrm{im}(f)$. \hfill $\square$
\end{lem}

\noindent
So every homomorphism $A\to B$ between von Neumann algebras admits a canonical factorisation as a projection onto a direct summand, followed by an injective homomorphism.

\begin{defn}
A \emph{left module} (\emph{right module}) over a von Neumann algebra $A$ is a Hilbert space $H$ equipped with a homomorphism $A\to B(H)$ ($A^{\op}\to B(H)$).
An \emph{$A$-$B$-bimodule} is a Hilbert space $H$ equipped with two homomorphisms $A\to B(H)$, $B^{\op}\to B(H)$ whose images commute in $B(H)$.

We write ${}_AH$ ($H_A$) to indicate that $H$ is a left module (right module), and ${}_AH_B$ to indicate that $H$ is an $A$-$B$-bimodule.
\end{defn}

The following is a fundamental result about the representation theory of von Neumann algebras:

\begin{prop}
Let $A$ be a von Neumann algebra, and $H$ a faithful $A$-module. 
Then for every non-zero $A$-module $K$,
$\Hom(K,H)\neq 0$.
\hfill $\square$
\end{prop}

\noindent
By polar decomposition and an application of Zorn's Lemma, we then have:

\begin{cor}\label{cor: submodules of opplus of faithful A-module}
Let $A$ be a von Neumann algebra, and $H$ a faithful $A$-module. 
Then for every $A$-module $K$ there exists a set $I$ such that $K$ is isomorphic to a submodule of $H^{\oplus I}$.
We may furthermore arrange for the inclusion $K\hookrightarrow H^{\oplus I}$
to be of the form
\[
\begin{tikzcd}
K=\bigoplus_{i\in I} K_i\arrow[r, hook, "\,\,\,\oplus_{i} f_i\,\,\,"] &
\bigoplus_{i\in I} H
\end{tikzcd}
\]
where each $f_i:K_i\hookrightarrow H$ is an isometric inclusion.
\hfill $\square$
\end{cor}

An abstract von Neumann algebra, also known as a \emph{$\rmW^*$-algebra}, is a $*$-algebra isomorphic to a von Neumann algebra on some Hilbert space.
By the celebrated work of Haagerup \cite{MR0407615}, every abstract von Neumann algebra $A$ admits a canonically associated Hilbert space $L^2A$ on which it acts, called its \emph{standard form},\label{pageref Standard Form} or non-commutative $L^2$-space. The space $L^2A$ is in fact an $A$-$A$-bimodule, and comes equipped with a canonical antiunitary involution
\begin{equation}\label{eq: mod conj}
J:L^2A\to L^2A 
\end{equation}
called \emph{modular conjugation}.
The modular conjugation satisfies $J(a\xi b) = b^* J(\xi) a^*$
for all $a,b\in A$ and $\xi \in L^2A$.
Importantly, the left and right actions of $A$ on $L^2A$ are each other's commutants.

The $L^2$ construction is compatible with the operation of taking corners:

\begin{lem}[{\cite[Lem~2.6]{MR0407615}}]
If $p\in A$ is a projection, then there is a canonical unitary $pL^2(A)p = L^2(pAp)$, compatible with the modular conjugations, and with left and right actions of $pAp$.
\hfill $\square$
\end{lem}

\begin{lem}
If $p,q\in A$ are projections, then $pL^2(A)q = 0$ if and only if $pAq = 0$.
\hfill $\square$
\end{lem}

The construction $A\mapsto L^2A$ is functorial for isomorphisms of von Neumann algebras, but not for general morphisms. It satisfies:

\begin{lem}
If $u\in A$ is a unitary, then the operators
$L^2(\mathrm{Ad}(u))$ and 
$\xi\mapsto u\xi u^*$
on $L^2(A)$ agree.
\hfill $\square$
\end{lem}

We write $\lambda:A\to B(L^2A)$ and $\rho: A^{\op}\to B(L^2A)$ for the two actions of $A$ on $L^2A$.
\begin{defn}[\cite{MR703809,MR1303779}]\label{def: Connes fusion}
If $H$ is a right $A$-module, and $K$ is a left $A$-module, their \emph{Connes fusion} $H\boxtimes_A K$ is the completion of the vector space
\[
\Hom_A(L^2A,H)\otimes_A L^2A\otimes_A \Hom_A(L^2A,K)
\]
with respect to the inner product
$
\langle \varphi_1\otimes \xi_1\otimes \psi_1,
\varphi_2\otimes \xi_2\otimes \psi_2\rangle
:=
\langle \xi_1,
\lambda^{-1}(\varphi_1^*\circ\varphi_2)\xi_2\rho^{-1}(\psi_1^*\circ \psi_2)\rangle
$.
\end{defn}

\begin{rem}
The Connes fusion of $H_A$ and ${}_AK$ also admits two asymmetric descriptions, as completions of $H\otimes_A \Hom_A(L^2A,K)$ and of $\Hom_A(L^2A,H)\otimes_A K$, respectively.
\end{rem}

The collection of all von Neumann algebras and all bimodules for von Neumann algebras forms the objects and $1$-morphisms of a bicategory. Composition of $1$-morphisms is given by Connes fusion, and the unit $1$-morphism on the object $A$ is the bimodule ${}_AL^2A_A$.
As part of the data of this bicategory,
there exist unitary associators and left and right unitors:
\begin{equation}\label{eq: associator + unitors}
\begin{split}
{}_A(H\boxtimes_BK)\boxtimes_CL_D
&\cong
{}_AH\boxtimes_B(K\boxtimes_CL)_D,
\\
{}_AL^2A\boxtimes_A H_B\cong {}_AH_B,\quad&\text{and}\quad {}_AH\boxtimes_BL^2B_B \cong {}_AH_B.
\end{split}
\end{equation}
We refer the reader to \cite{10.1063/1.1563733,MR2111973,MR4419534} for details about this bicategory.

\section{\texorpdfstring{$\rmW^*$}{W*}-categories: basic properties and operations}

A \emph{$*$-category} is a $\C$-linear category equipped with a dagger structure $*:\Hom(X,Y)\to \Hom(Y,X)$ which is $\C$-antilinear, and satisfies $f^{**}=f$ and $(f\circ g)^*=g^*\circ f^*$.
An invertible morphism in such a category is called \emph{unitary} if $f^*=f^{-1}$.

Given a $*$-category $C$, we write $\overline{C}$ 
\label{pageref: complex conjugate}
for the category with same objects, and complex conjugate hom-spaces.
The $*$-operation provides a canonical identification between $\overline{C}$ and $C^{\op}$.
Given a $*$-category $C$, we write $C^{\oplus}$
\label{pageref: completion under finite direct sums}
for the category whose objects are formal finite sums of objects of $C$, and whose morphisms are given by $\Hom(\oplus c_i,\oplus c_j):=\oplus_{i,j}\Hom(c_i,c_j)$.

\begin{defn}\label{def: W-star cat}
A \emph{$\rmW^*$-category} is a $*$-category $C$ such that $\End(X)$ is a $\rmW^*$-algebra for every $X\in C^{\oplus}$.\footnote{See \cite[Cor.~1.2]{MR2325696} for the equivalence between this definition, and other definitions of $\rmW^*$-category that exist in the literature.} The objects of a $\rmW^*$-category may form a proper class, but the hom-spaces between any two objects are required to be set-sized.
We refer the reader to \cite{MR808930} for further details.
\end{defn}

The norm of a morphism $f:c\to d$ in a $\rmW^*$-category is denoted $\|f\|$, and is defined to be the norm of $f$ in the 
$\rmW^*$-algebra $\End_{C^\oplus}(c\oplus d)$.

\begin{exs*} The following are examples of $\rmW^*$-categories:
\begin{itemize}
\item\label{category Hilb}
$\Hilb$, the category of Hilbert spaces and bounded linear maps.

\item\label{category AMod}
The category $A$-$\Mod$ of normal modules
over some von Neumann algebra $A$.

\item
The category $\Rep(G)$ of unitary representations of a group $G$.

\item
The category of modules over a $\rmC^*$-algebra.

\item\label{category BA}
The category $\mathbf{B}A$ with exactly one object $\star_A$, and $\End(\star_A):=A$, where $A$ is some von Neumann algebra.
\end{itemize}
\end{exs*}

\noindent
Henceforth, all modules over von Neumann algebras will be assumed to be normal, even when this is not mentioned explicitly.

\begin{defn}
\label{def: functors linear and bilinear}
A functor between $\rmW^*$-categories is a $*$-functor
\[
F:C\to D
\]
that induces normal homomorphisms $\End(c)\to \End(F(c))$ for all $c \in C^{\oplus}$.
Similarly, a bilinear functor between $\rmW^*$-categories is a $*$-functor $F:C_1\times C_2\to D$ which is bilinear, and normal in each variable.

A natural transformation $\alpha:C\tworarrow D$ between functors $F,G:C\to D$ is called \emph{bounded} if $\|\alpha\|:=\sup_{c\in C} \|\alpha_c\|<\infty$.
From now on, all natural transformation will be assumed bounded, even when this is not mentioned explicitly.
\end{defn}

A functor $F:C\to D$ is called 
\emph{faithful} if for every pair of objects $c,c'\in C$, the map $\Hom(c,c')\to \Hom(F(c),F(c'))$ is injective. 
It is called \emph{fully faithful}\label{pageref: (fully) faithful functor} if that map is always an isomorphism.
The functor is called
\emph{dominant}\label{pageref: dominant functor} if every object of $D$ is a direct summand of an object of the form $F(c)$, for some $c\in C$.

\begin{defn}\label{def: idempotent complete}
A $\rmW^*$-category is called \emph{idempotent complete} if whenever a morphism $p:x\to x$ satisfies $p^2=p^*=p$,
there exists an object $y$ and a morphism $\iota:y\to x$ such that $\iota\iota^*=p$ and $\iota^*\iota=\id_y$.
In the above situation, we say that $y$ splits the idempotent $p$.
(This condition is called \emph{having sufficient subobjects} in \cite{MR808930}.)
\end{defn}

\begin{lem}
A functor $F:C\to D$ between idempotent complete $\rmW^*$-categories is an equivalence if and only if it is fully faithful and dominant.
\end{lem}

\begin{proof}
We show that if $F$ is fully faithful and dominant, then it is essentially surjective. 
Let $d\in D$ be an object.
Since $F$ is dominant, there exists $c\in C$ such that $d$ is a direct summand of $F(c)$.
Let $p:F(c)\to F(c)$ be the orthogonal projection onto that summand.
Since $F$ is fully faithful, there exists a projection $q:c\to c$ such that $F(q)=p$.
Let $c'\in C$ be the object that splits the idempotent $q$. Then $F(c')\cong d$.
\end{proof}

The \emph{idempotent completion}
\label{page ref: idempotent completion}
of a $\rmW^*$-category $C$ is the category $\hat C$ whose objects are pairs $(x,p:x\to x)$ with $p^2=p^*=p$, and whose morphisms are given by
\[
\Hom_{\hat C}((x_1,p_1),(x_2,p_2)):=p_2 \Hom_C(x_1,x_2) p_1.
\]
The idempotent completion of an idempotent complete category is always equivalent to the original category.

\begin{defn}\label{def: admit a set of generators}
A $\rmW^*$-category $C$ is said to \emph{admit a set of generators} if there exists a set of objects such that every non-zero object of $\hat C$ admits a non-zero map from at least one of the generators.
It is said to \emph{admit a generator} if the above set may be taken to consist of a single object (see \cite[Prop 7.3]{MR808930} for an equivalent characterisation of generators).
\end{defn}

\begin{exs*}$\,$\vspace{-1mm}
\begin{itemize}
\item
If $A$ is a von Neumann algebra, then
any faithful $A$-module generates $A$-$\Mod$.
For example, ${}_AL^2A$ is a generator.

\item
The set of all unitary representations of $G$ on $\ell^2G$ (not just the regular one) forms a set of generators for $\Rep(G)$.

\item
Let $C$ be the category whose objects are collections of Hilbert spaces indexed by the class of all ordinals, where the Hilbert spaces are required to be zero beyond a certain ordinal. Then $C$ is a $\rmW^*$-category which does not admit a set of generators.
\end{itemize}
\end{exs*}

\begin{defn}
\label{def: functor categories}
Let $C$ and $D$ be $\rmW^*$-categories. If $C$ admits a set of generators, then the collection of all functors $C\to D$, and all (bounded) natural transformations $C\tworarrow D$ form the objects and morphisms of a $\rmW^*$-category
\[
\mathrm{Func}(C,D).
\]
Given a natural transformation $\alpha:F\Rightarrow G$ between functors $F,G\in \mathrm{Func}(C,D)$, then $\alpha^*:G\Rightarrow F$ is defined pointwise, by the formula $(\alpha^*)_c:=(\alpha_c)^*$.
\end{defn}

Note that if $C$ doesn't have a set of generators, then the functors $C\to D$ might be too numerous, and thus fail to form a class.
Similarly, natural transformations between two functors might fail to form a set.

From now on, all $\rmW^*$-categories in this paper will be assumed to have a set of generators.\medskip

\noindent
\label{page technical remark}
{\bf Technical remark.}\, \it When $C$ is a large $\rmW^*$-category, even if $C$ admits a set of generators, the collection of all functors $C\to D$ is typically too numerous, and thus fails to form a class. 
A way to address this set-theoretic difficulty is to pick a full subcategory $C_0\subset C$ on some set of generators
and replace $\mathrm{Func}(C,D)$
by the full subcategory of $\mathrm{Func}(C_0,D)$ on those functors $C_0\to D$ that admit an extension $C\to D$.
Such an extension is essentially unique if it exists (see Lemma~\ref{lem: Func(C,D) is a full subcategory} below), so
this trick produces a category which is equivalent to the very large category of all functors $C\to D$, but whose objects nevertheless do form a class.
\rm \medskip

If $C_1$, $C_2$, $D$ are $\rmW^*$-categories, then currying provides an equivalence\label{pageref: biliar functors}
\[
\mathrm{Bilin}\big(C_1\times C_2,D\big)
\cong
\mathrm{Func}\big(C_1,\mathrm{Func}(C_2,D)\big)
\]
between blinear functors $C_1\times C_2\to D$ and functors $C_1\to \mathrm{Func}(C_2,D)$.

\begin{defn}
[{\cite[p.100]{MR808930}}]\label{def: orthogonal direct sum}
Given a collection of objects $c_i$ in a $\rmW^*$-category indexed by some set $I$, their \emph{orthogonal direct sum} is an object $x$ equipped with morphisms $\iota_i:c_i\to x$ satisfying
\begin{equation}
\label{eq: dir sum}
\iota_i^* \iota_j=\delta_{ij}\id_{c_i}\qquad\qquad \sum \iota_i^* \iota_i=\id_x.
\end{equation}
The orthogonal direct sums, if it exists, is denoted $\oplus_{i\in I} c_i$.
Here, the infinite sum $\sum_{i\in I} \iota_i^* \iota_i$ in \eqref{eq: dir sum} is defined as the sup over all finite subsets $I_0\subset I$ of the finite sums $\sum_{i\in I_0} \iota_i^* \iota_i$.
\end{defn}

Orthogonal direct sums, if they exist, are unique up to unique unitary isomorphism.
We refer the reader to \cite[Def.~4.2]{MR4077245}
for a universal property satisfied by orthogonal direct sums.
As a corollary of the universal property, if $\jmath_i:c_i\to y$ satisfy $\jmath_i^*\jmath_j=\delta_{ij}\id_{c_i}$, then they induce an isometry $\oplus c_i\to y$ -- this also follows from part {\it(ii)} of Proposition~\ref{prop:freydembedding}, below.

Orthogonal direct sums are universal in the following sense:
\begin{lem}
If $F:C\to D$ is a functor between $\rmW^*$-categories, and $c_i\in C$ is a collection of objects that admits an orthogonal direct sum, then so do their images under $F$, and there is a canonical unitary
$F(\oplus_i c_i) \cong \oplus_i F( c_i)$. \hfill $\square$
\end{lem}

\begin{proof}
If $\iota_i:c_i\to \oplus_i c_i$ are the inclusions, then $F(\iota_i)$ satisfy \eqref{eq: dir sum}, hence exhibit $F(\oplus c_i)$ as $\oplus F(c_i)$.
\end{proof}

A $\rmW^*$-category is said to \emph{admit all direct sums}\label{pageref: to admit all direct sums} if every collection of objects admits an orthogonal direct sum.
Note that a functor between $\rmW^*$-categories automatically preserves orthogonal direct sums (because functors are required, by definition, to induce normal homomorphisms between endomorphism algebras).

The following construction was hinted in \cite[p.100]{MR808930}.
Let us write $I_0\Subset I$ to indicate that $I_0$ is a finite subset of some set $I$.
Let $C$ be a $\rmW^*$-category.
The \emph{direct sum completion} of $C$ is the $\rmW^*$-category $C^{\bar\oplus}$\label{pageref: direct sum completion} whose objects are formal direct sums $\oplus_{i\in I} x_i$ of objects of $C$, indexed over arbitrary sets.
Its morphisms are defined as follows.
We first let $\End_{C^{\bar \oplus}}(\oplus_{i\in I} x_i)$
be the von Neumann algebra generated by 
${\oplus}_{i,j\in I}\Hom(x_i,x_j) = \bigcup_{I_0\Subset I}
\End_{C^\oplus}(\oplus_{i\in I_0} x_i)$ 
inside the algebra of bounded operators on the Hilbert space completion of the pre-Hilbert space
$\bigcup_{I_0\Subset I}
L^2(\End_{C^\oplus}(\oplus_{i\in I_0} x_i))$.
We then define
\[
\textstyle
\Hom_{C^{\bar \oplus}}\!\big(\!\oplus_{i\in I} x_i,\oplus_{j\in J} x_j\big) :=\,
p_2
\End_{C^{\bar \oplus}}\!\big(\!\oplus_{k\in I\sqcup J} x_k\big)p_1,
\]
where $p_1$ and $p_2$ are the projections onto the subobjects $\oplus_{i\in I} x_i$ and $\oplus_{j\in J} x_j$ of $\oplus_{k\in I\sqcup J} x_k$.

\begin{lem}
\label{lem: if C admits all direct sums, then so does its idempotent completion}
If $C$ admits all direct sums, then so does its idempotent completion $\hat{C}$.
\end{lem}
\begin{proof}
Given a family $(x_i,p_i)\in\hat C$, the sum  $\sum p_i$ is a projection in $\End(\oplus x_i)$.
One readily checks that $(\oplus x_i, \sum p_i)$ is an orthogonal direct sum of the $(x_i,p_i)$.
\end{proof}

\begin{defn}\label{def: Cauchy completion}
A $\rmW^*$-category is said to be \emph{Cauchy complete} if it admits all direct sums, and if it is idempotent complete.
We write $C^{\hat \oplus}:=
(\hat C)^{\bar\oplus}$
for the direct sum completion of the idempotent completion 
of a $\rmW^*$-category $C$, and call it
the \emph{Cauchy completion} of $C$.
(See Corollary~\ref{cor: C hat oplus = c oplus hat} for an equivalent definition of $C^{\hat \oplus}$, and the fact that the Cauchy completion of $C$ is indeed Cauchy complete.)
\end{defn}

\begin{term}\label{rem: terminology complete}
We propose that the term `complete $\rmW^*$-category' be used synonymously to `Cauchy complete $\rmW^*$-category' (as in the title of this paper). However, given the potential for confusion with other types of completeness ---idempotent complete, direct sum complete--- we will refrain from using that terminology in this section.
\end{term}

Note that a set of objects generates $C$ iff it generates $C^{\hat \oplus}$.

\begin{exs*}$\,$\vspace{-1mm}
\begin{itemize}
\item
$(\mathbf{B}\bbC)^{\hat \oplus}=(\mathbf{B}\bbC)^{\bar \oplus}=\Hilb$.

\item
If $A$ is a von Neumann algebra, then the map $\star_{A^{\op}}\mapsto {}_AL^2A$ extends to an equivalence $(\mathbf{B}A^{\op})^{\hat \oplus}\cong A\text{-}\Mod$. (See Lemma~\ref{lem: canonical  functor an equivalence iff C_0 generates C}.)

\item
If $D$ is Cauchy complete, then so is $\mathrm{Func}(C,D)$.
\end{itemize}
\end{exs*}

\noindent
The next result is very well known (see \cite[Thm.7.13]{MR808930} for a special case):

\begin{prop}
\label{prop: idempotent completion}
Let $C$ be a $\rmW^*$-category, with idempotent 
completion $\hat C$, and let $D$ be an idempotent complete $\rmW^*$-category. 
Then the restriction functor
$
r:\mathrm{Func}(\hat C,D) \to \mathrm{Func}(C,D)
$
is an equivalence of categories.
\end{prop}
\begin{proof}
\underline{$r$ is essentiall}y\!\underline{\, sur}j\underline{ective:}\,
Given a functor $F:C\to D$, we construct an extension $\hat F:\hat C\to D$
by leting $\hat{F}(c,p)$ be any object that splits $F(p)$.
The\vspace{-1.2mm}
image under $\hat{F}$ of a morphisms $f:(c_1,p_1)\to (c_2,p_2)$  is the composite $\hat F(c_1,p_1)\hookrightarrow F(c_1)\stackrel{F(f)}\rightarrow F(c_2)\twoheadrightarrow F(c_2,p_2)$.\medskip

\noindent
\underline{$r$ is full:}\,
Given a natural transformation $\alpha: F\Rightarrow G$,
we construct an extension $\hat\alpha:\hat F\Rightarrow \hat G$
by letting $\hat\alpha_{(c,p)}$
be the composite
$\hat F(c,p)\hookrightarrow F(c)\stackrel{\alpha_c}\rightarrow G(c)\twoheadrightarrow G(c,p)$.\medskip

\noindent
\underline{$r$ is faithful:}\,
If $\beta: \hat{F}\Rightarrow \hat{G}$ satisfies $\beta|_C=0$, then $\beta_{(c,p)}:\hat F(c,p) \to \hat G(c,p)$ is zero $\forall (c,p)\in\hat C$ by naturality with respect to the inclusions $(c,p)\hookrightarrow (c,0)$.
\end{proof}

The next result is quite special to the theory of $\rmW^*$-categories, and is a big part of what makes the whole theory work:

\begin{prop}
\label{prop: direct sum completion}
Let $C$ be a $\rmW^*$-category, and let $D$ be a $\rmW^*$-category which admits all direct sums. Then the restriction functor
\begin{equation}
\label{eq: restrictionfunctor}
\mathrm{Func}(C^{\bar\oplus},D) \to \mathrm{Func}(C,D)
\end{equation}
is an equivalence of categories.
\end{prop}
\begin{proof}
We first show that
\eqref{eq: restrictionfunctor} is essentially surjective.
Given a functor $F: C\to D$, 
we define an extension $F^+: C^{\bar\oplus}\to D$ of $F$ as follows.
At the level of objects,
$F^+(\oplus c_i):=\oplus F(c_i)$.
At the level of endomorphism algebras, we let
\[
\textstyle
F^+:\End(\oplus_{i\in I} c_i) \to \End(\oplus_{i\in I} F(c_i))
\]
be the unique extension, by Lemma~\ref{lem: homomorphism defined on union of pMp}, of
\begin{align*}
\bigcup_{I_0\Subset I}
\End(\oplus_{i\in {I_0}} c_i)
\,&\to\,
\bigcup_{I_0\Subset I}
\End(\oplus_{i\in {I_0}} F(c_i))
\,\hookrightarrow\,
\End(\oplus_{i\in I} F(c_i)),
\\
(f_{ij})_{i,j \in I_0}
\quad
&\mapsto\quad \sum_{i,j\in I_0} \iota^*_j\, F(f_{ij})\, \iota_i
\end{align*}
where $\iota_j:F(c_j) \to \oplus_{i\in {I_0}} F(c_i)$ is the inclusion of the $j$-th summand.
Here, as before, we have used the notation $I_0\Subset I$ to mean that $I_0\subset I$ and $|I_0|<\infty$.
Finally, at the level of hom-spaces, we let
\(
\textstyle
F^+:\Hom(\oplus_{i\in I} c_i, \oplus_{j\in J} c_j) \to \Hom(\oplus_{i\in I} F(c_i),\oplus_{j\in J} F(c_j))
\)
be the appropriate corner of
$
\textstyle
F^+:\End(\oplus_{k\in I\cup J} c_k) \to \End(\oplus_{k\in I\cup J} F(c_k))
$.

To show that \eqref{eq: restrictionfunctor} is full,
we argue that every (bounded) natural transformation $\alpha: F\Rightarrow G$ between functors $C\to D$
extends to a natural transformation $\alpha^+: F^+\Rightarrow G^+$.
Letting 
$\iota^F_j:F(c_j) \to \oplus_{i\in I} F(c_i)$ and
$\iota^G_j:G(c_j) \to \oplus_{i\in I} G(c_i)$
denote the inclusions,
$\alpha^+$ is given by
\begin{equation*}
(\alpha^+)_{\oplus_{i\in I} c_i}:= \sum_{i\in I} \iota^G_i\circ \alpha_{c_i}\circ (\iota^F_i)^* : \oplus_{i\in I} F(c_i)\to \oplus_{i\in I} G(c_i).
\end{equation*}
The sum is convergent because $\alpha$ is bounded.

Finally, if $\beta: F^+\Rightarrow G^+$ restricts to zero on $C\subset C^{\bar\oplus}$, then each composite
$\beta_{\oplus c_i}\circ \iota^F_j: F(c_j) \to \oplus_i F(c_i)\to \oplus_i G(c_i)$
is equal to $\iota^G_j\circ \beta_{c_j}$ by naturality, hence is zero.
So $\beta^+$ is zero, and \eqref{eq: restrictionfunctor} is faithful.
\end{proof}

As a formal consequence of Proposition~\ref{prop: direct sum completion},
the direct sum completion of a category which admits all direct sums is always equivalent to the original category.
Combining Propositions~\ref{prop: idempotent completion} and~\ref{prop: direct sum completion}, we get a version of \cite[Thm~7.13]{MR808930}:

\begin{cor}
\label{cor: Cauchy completion}
Let $C$ and $D$ be $\rmW^*$-categories, with $D$ Cauchy complete.
Then the restriction functor
\[
\mathrm{Func}\big(C^{\hat \oplus},D\big) \to \mathrm{Func}\big(C,D\big)
\]
is an equivalence of categories.\hfill $\square$
\end{cor}

Propositions~\ref{prop: idempotent completion},~\ref{prop: direct sum completion}, and Corollary~\ref{cor: Cauchy completion} admit straightforward generalisations to the case of bilinear functors.
For example:

\begin{cor}
\label{cor: Cauchy completion'}
Let $C_1$, $C_2$ and $D$ be $\rmW^*$-categories, with $D$ Cauchy complete.
Then the restriction functor
\[
\mathrm{Bilin}\big(C_1^{\hat \oplus}\times C_2^{\hat \oplus},D\big) \to \mathrm{Bilin}\big(C_1\times C_2,D\big)
\]
is an equivalence of categories.
\end{cor}

\begin{proof}
If $D$ is Cauchy complete,\, then so is $\mathrm{Func}(C,D)$.\, By Corollary~\ref{cor: Cauchy completion},\, we then have:\vspace{-3mm}
\[
\mathrm{Bilin}\big(C_1^{\hat \oplus}\times C_2^{\hat \oplus},D\big)
=
\mathrm{Func}\big(C_1^{\hat \oplus},\mathrm{Func}(C_2^{\hat \oplus},D)\big)
=
\mathrm{Func}\big(C_1,\mathrm{Func}(C_2,D)\big)
=\mathrm{Bilin}\big(C_1\times C_2,D\big).
\qedhere
\]
\end{proof}

\begin{lem}
\label{lem: canonical  functor an equivalence iff C_0 generates C}
Let $C$ be a Cauchy complete $\rmW^*$-category, and let $C_0\subset C$ be the full subcategory on some set of objects.
Then the  functor $F:(C_0)^{\hat \oplus}\to C$ provided by Corollary~\ref{cor: Cauchy completion} is fully faithful. It is an equivalence if and only if $C_0$ generates $C$.
\end{lem}
\begin{proof}
Since $C_0\hookrightarrow C$ is fully faithful, so is
$C_0^{\hat \oplus}\hookrightarrow C^{\hat \oplus}$.
The functor $F$ factors as $C_0^{\hat \oplus}\hookrightarrow C^{\hat \oplus}\cong C$, so it is also fully faithful.

We now show that if $C_0$ generates $C$, then
$F$ is essentially surjective.
Given $c\in C$, we may pick by Zorn's lemma a maximal family of non-zero partial isometries $\{v_i: x_i\to c\}$ with domains $x_i\in C_0$, and orthogonal ranges.
Let $p_i:=v_iv_i^*$ and $q_i:=v_i^*v_i$.
We have $\sum p_i=1$, as otherwise, by picking a non-zero morphism from some object $x\in C_0$ to the object $y$ that splits the idempotent $1-\sum p_j$, and 
applying polar decomposition, we would get a new non-zero partial isometry $x\to y \hookrightarrow c$ whose range is orthogonal to the ones in our family, contradicting the maximality of the family.
So $\sum p_i=1$, and the $v_i$ induce a unitary from  $F(\oplus_i(x_i,q_i))$ to $c$. 

Finally, since $C_0$ generates its Cauchy completion, if the functor $(C_0)^{\hat \oplus}\to C$ is an equivalence, then $C_0$ generates $C$.
\end{proof}

We are now in position to justify the technical remark on page \pageref{page technical remark}:

\begin{lem}[{\cite[Lem~2.2]{MR2325696}}]
\label{lem: Func(C,D) is a full subcategory}
Let $C$ and $D$ be $\rmW^*$-categories,
and let $C_0\subset C$ be the 
full subcategory on some set of generators.
Then $\mathrm{Func}(C,D)$ is equivalent to the full subcategory of $\mathrm{Func}(C_0,D)$ on those functors $C_0\to D$ that admit an extension $C\to D$.
\end{lem}

\begin{proof}
By Lemma~\ref{lem: canonical  functor an equivalence iff C_0 generates C}, since $C_0$ generates $C^{\hat \oplus}$, we have
$(C_0)^{\hat \oplus}\cong C^{\hat \oplus}$.
The functors
\[
\mathrm{Func}(C,D) \to 
\mathrm{Func}(C,D^{\hat\oplus})
\cong
(C^{\hat\oplus},D^{\hat\oplus})
\cong
\mathrm{Func}(C_0,D^{\hat\oplus})
\leftarrow
\mathrm{Func}(C_0,D)
\]
are fully faithful
(were the middle equivalences hold by Corollary~\ref{cor: Cauchy completion}), so $\mathrm{Func}(C,D)\to \mathrm{Func}(C_0,D)$ is fully faithful.
Its essential image consists of those functors $C_0\to D$ that admit an extension $C\to D$.
\end{proof}

Recall that our $\rmW^*$-categories are assumed to all admit a set of generators.
As an immediate consequence of 
Lemma~\ref{lem: canonical  functor an equivalence iff C_0 generates C},
we get the following important structural result about Cauchy complete $\rmW^*$-categories:

\begin{prop}[{\cite[Prop.~7.6]{MR808930}}]
\label{prop:freydembedding}
\mbox{}
\begin{enumerate}[label=(\roman*)]
    \item 
    Every Cauchy complete $\rmW^*$-category $C$ is equivalent to the category of modules over some von Neumann algebra.
    \item 
    Every $\rmW^*$-category is equivalent to a full subcategory of the category of modules over some von Neumann algebra.
\end{enumerate}
\end{prop}

\begin{proof}
{\it (i)} Let $c \in C$ be a generator, let $A:=\End (c)^{\op}$,
and let $\mathbf BA^{\op}\to C$ and $\mathbf BA^{\op}\to A\text{-}\Mod$ be the functors given by $\star_{A^{\op}}\mapsto c$ and $\star_{A^{\op}}\mapsto {}_AL^2A$, respectively.
By Lemma~\ref{lem: canonical  functor an equivalence iff C_0 generates C},
these extend to equivalences $\mathbf (BA^{\op})^{\hat\oplus}\to C$ and $(BA^{\op})\to A\text{-}\Mod$.
So $C\cong A\text{-}\Mod$.

{\it (ii)} Any $\rmW^*$-category is a full subcategory of its Cauchy completion. By part {\it (i)}, the latter is equivalent to the category of modules over some von Neumann algebra.
\end{proof}

In \eqref{eq: freydembedding'} below, we will refine the above result 
by providing an explicit formula for the functor 
$C\to A\text{-}\Mod$ whose existence is guaranteed by the above proposition.

The following analog of the Cantor-Schr\"oder-Bernstein theorem is very well known:

\begin{lem}\label{lem:Cantor-Schroeder-Bernstein}
If there exist isometries $a\to b$ and $b\to a$ in a $\rmW^*$-category, then $a\cong b$.
\end{lem}
\begin{proof}
By Proposition~\ref{prop:freydembedding}, we may realise $C$ as a full subcategory of $A$-$\Mod$, for some von Neumann algebra $A$.
Let ${}_AH$ and ${}_AK$ be the $A$-modules corresponding to $a$ and $b$,
with isometries $f: H \to K$ and $g: K \to H$. Setting
\begin{align*}
H_0 := H&  &H_{n+1} := g(K_n)&  &H^n := H_n \ominus H_{n-1}&  &H_\infty := \bigcap H_n\\
K_0 := K&  &K_{n+1} := f(H_n)&  &K^n := K_n \ominus K_{n-1}& &K_\infty := \bigcap K_n
\end{align*}
we have
$H  =  H_\infty \oplus H^0 \oplus H^1 \oplus H^2 \oplus \ldots$ and
$K  =  K_\infty \oplus K^0 \oplus K^1 \oplus K^2 \oplus \ldots$
The desired isomorphism $H \to K$ is the direct sum of the isomorphisms $f: H^i \to K^{i+1}$ for $i$ even, the inverses of $g: K^{i-1} \to H^{i}$ for $i$ odd, and $f: H_\infty \to K_\infty$.
\end{proof}

\begin{defn}\label{def: full subcategory generated by}
Given an object $c\in C$ of some Cauchy complete $\rmW^*$-category, we define the \emph{full subcategory of $C$ generated by $c$} to be the essential image of the functor $\mathbf B\End(c)^{\hat \oplus}\to C$ provided by Corollary~\ref{cor: Cauchy completion}.

Similarly, we define the full subcategory of $C$ generated by some collection $\{c_i\}$ of objects to be the full subcategory generated by $\oplus_i c$.
\end{defn}

\begin{cor}
\label{cor: C hat oplus = c oplus hat}
If $C$ is a $\rmW^*$-category, then the canonical functor $C^{\hat\oplus}=(\hat C)^{\bar\oplus}\to (C^{\bar\oplus})\hspace{-.5mm}\hat{\phantom t}:
\oplus_i(x_i,p_i)\mapsto(\oplus_ix_i,\sum p_i)$
is an equivalence of categories.
\end{cor}
\begin{proof}
The category 
$(C^{\bar\oplus})\hspace{-.5mm}\hat{\phantom t}$ is Cauchy complete by Lemma~\ref{lem: if C admits all direct sums, then so does its idempotent completion}. Now
apply Lemma~\ref{lem: canonical  functor an equivalence iff C_0 generates C} 
to its full subcategory $C\hookrightarrow (C^{\bar\oplus})\hspace{-.5mm}\hat{\phantom t}$.
\end{proof}

\begin{siderem*}
Applying Corollary~\ref{cor: C hat oplus = c oplus hat} to the $\rmW^*$-category $\mathbf BA$ for some von Neumann algebra $A$, we learn that every projection in the matrix algebra $M_n(A)$ is equivalent to a direct sum of projections in $A$. It is interesting to note that $\rmC^*$-algebras typically do not have that property (for example, $C(S^2)$ does not have that property). Those $\rmC^*$-algebras that do are said to satisfy $K_0$-surjectivity \cite{MR3206515}
\end{siderem*}

\begin{cor}
\label{cor: C full subcategory of C_0 plus}
If $C$ is a $\rmW^*$-category and $C_0\hookrightarrow C$ is the full subcategory on some set of generators, then $C$ is a full subcategory of $(C_0)^{\hat \oplus}$.
\end{cor}
\begin{proof}
Since $C_0$ generates $C$, it also generates $C^{\hat \oplus}$,
so we have 
$C \hookrightarrow C^{\hat \oplus}\cong (C_0)^{\hat \oplus}$,
where the last equivalence is provided by Lemma~\ref{lem: canonical  functor an equivalence iff C_0 generates C}.
\end{proof}

\begin{cor}
\label{cor: equivalence checkable on generators}
Let $F:C\to D$ be a functor between Cauchy complete $\rmW^*$-categories.
If $c$ generates $C$,
if $d:=F(c)$ generates $D$, and
if $F:\End(c)\to \End(d)$ is an isomorphism,
then $F$ is an equivalence of categories.
\end{cor}

\begin{proof}
Let $A:=\End (c)$, and let $\mathbf BA\to C$ and $\mathbf BA\to D$ be the functors given by $\star_A\mapsto c$ and $\star_A\mapsto d$, respectively.
These induce equivalences $\mathbf BA^{\hat\oplus}\to C$ and $\mathbf BA^{\hat\oplus}\to D$ by Lemma~\ref{lem: canonical  functor an equivalence iff C_0 generates C}.
The diagram 
\[
\begin{tikzcd}[baseline=-27]
&\mathbf BA^{\hat \oplus}\!
\arrow[dr,"\cong"]
\arrow[dl,"\cong"']
&
\\
C\arrow[rr,"F"]
&&D
\end{tikzcd}
\]
commutes when precomposed with the inclusion $\mathbf BA\hookrightarrow \mathbf BA^{\hat \oplus}$, hence commutes (up to an invertible $2$-morphism) by Corollary~\ref{cor: Cauchy completion}.
So $F$ is an equivalence.
\end{proof}

\begin{cor}
\label{cor: full subcat same functors}
Let $C$ and $D$ be $\rmW^*$-categories, with $D$ Cauchy complete, and let $C_0\subset C$ be the full subcategory on some set of generators. Then the restriction functor
$\mathrm{Func}(C,D) \to \mathrm{Func}(C_0,D)$
is an equivalence of categories.
\end{cor}

\begin{proof}
We have $\mathrm{Func}\big(C,D\big) \cong
\mathrm{Func}\big(C^{\hat \oplus},D\big) \cong \mathrm{Func}\big(C_0^{\hat \oplus},D\big) \cong \mathrm{Func}\big(C_0,D\big)$,
where the first and last equivalences are provided by Corollary~\ref{cor: Cauchy completion},
and the middle equivalence is provided by 
Lemma~\ref{lem: canonical  functor an equivalence iff C_0 generates C}.
\end{proof}

The most important instance of Corollary~\ref{cor: full subcat same functors} is when $C_0$ is the full subcategory on a single object $c\in C$, in which case $C_0\cong \mathbf B A$, for $A=\End(c)$. By Corollary~\ref{cor: full subcat same functors}, any functor $\mathbf B A \to D$ admits a unique up to contractible choice extension to a functor $C\to D$, and any natural transformation $\mathbf BA\tworarrow D$ admits a unique extension to a natural transformation $C\tworarrow D$.
Since the data of a functor $\mathbf B A \to D$ is equivalent to the data of an object $D$ equipped with an action of $A$,
and the data of a natural transformation $\mathbf BA\tworarrow D$ is equivalent to the data of an $A$-equivariant map between two $A$-objects of $D$, we may rephrase the result as follows:

\begin{cor}
\label{cor: enough to construct on generator}
Let $C$ and $D$ be $\rmW^*$-categories. Assume that $D$ is Cauchy complete, and that $c\in C$ generates $C$. Let $A:=\End_C(c)$. Then:
\begin{enumerate}[label=(\arabic*)]
\item
Providing an object $d\in D$ with an action $A\to \End_D(d)$ is equivalent, up to contractible choice, to providing a functor $C\to D$.
\item
Providing an $A$-equivariant map $d_1\to d_2$ between $A$-objects $d_1,d_2\in D$ is equivalent to providing a natural transformation $C\tworarrow D$ between the functors associated to $d_1$ and~$d_2$.
\end{enumerate}
\end{cor}

\noindent
In equations \eqref{eq: formula for unique extension} and \eqref{eq: formula for unique extension 2} further down, we will provide explicit formulas for the extensions whose existence and uniqueness is guaranteed by Corollary~\ref{cor: enough to construct on generator}.

\begin{defn}\label{def: stably equivalent objects}
Two objects $c, d\in C$ of a $\rmW^*$-category are called \emph{stably equivalent} if they generate the same full subcategory of $C^{\hat \oplus}$. We write $c \cong_{st} d$ to indicate that $c$ and $d$ are stably equivalent.
\end{defn}

\begin{lem}
Two objects $c,d\in C$ are stably equivalent if and only if for every set $I$ of sufficiently large cardinality, we have $c^{\oplus I}\cong d^{\oplus I}$ in $C^{\bar \oplus}$.
\end{lem}

\begin{proof}
If $c^{\oplus I}\cong d^{\oplus I}$, then $c \cong_{st} c^{\oplus I}\cong d^{\oplus I} \cong_{st} d$, hence $c \cong_{st} d$.

If $c \cong_{st} d$, we may assume without loss of generality that $c$ and $d$ generate $C$, and that $C$ is Cauchy complete.
Upon identifying $C$ with $A$-$\Mod$ for some von Neumann algebra $A$, the objects $c$ and $d$ correspond to faithful $A$-modules. Call those ${}_AH$ and ${}_AK$.
Since $H$ is faithful, there exists a set $I$ and an $A$-linear isometry $K\hookrightarrow H^{\oplus I}$. Similarly, there exists a set $J$ and an $A$-linear isometry $H\hookrightarrow K^{\oplus J}$. Assuming without loss of generality that $I=J$ and that $I$ is infinite, we then get (using the well known fact that the cardinality of an infinite set is equal to that of its square):
\[
K^{\oplus I}\hookrightarrow (H^{\oplus I})^{\oplus I} = H^{\oplus I\times I} \cong H^{\oplus I} \qquad\text{and}\qquad 
H^{\oplus I}\hookrightarrow (K^{\oplus I})^{\oplus I} = K^{\oplus I\times I} \cong K^{\oplus I},
\]
from which it follows that $H^{\oplus I}\cong K^{\oplus I}$ by Lemma~\ref{lem:Cantor-Schroeder-Bernstein}.

The same argument holds for any set of cardinality larger than that of $I$. 
\end{proof}

In Proposition~\ref{prop:freydembedding}, we have seen that every Cauchy complete $\rmW^*$-category is equivalent the category of modules over some von Neumann algebra.
The next result characterises functors between such $\rmW^*$-categories:

\begin{prop}[{\cite[Prop~1.13]{MR2325696}}]
\label{prop: func to bimod}
Given von Neumann algebras $A$ and $B$, the functor \begin{align}
\label{eq: Func equivalence}
\Bim(B, A)
\,&\to\, \mathrm{Func}\big(
A\text{-}\Mod, B\text{-}\Mod\big)
\end{align}
which sends a bimodule 
${}_BX_A $ to the functor ${}_BX\boxtimes_A-$
is an equivalence of categories. The inverse of \eqref{eq: Func equivalence} sends a functor $F:A\text{-}\Mod\to B\text{-}\Mod$
to the bimodule
${}_B\big(F({}_AL^2A)\big){}_A$.
\end{prop}
\begin{proof}
Starting from ${}_BX_A\in\mathrm{Bim}(B, A)$,
one easily checks that
$F \mapsto {_B}(F({}_AL^2A))_A$
applied to 
$({}_BX\boxtimes_A-)$ recovers ${}_BX_A$.
Starting from $F:A\text{-}\Mod\to B\text{-}\Mod$,
we need to check that 
${_B}\big(F({}_AL^2A))\boxtimes_A-$
is naturally isomorphic to $F$.
By Corollary~\ref{cor: enough to construct on generator}, it is sufficient to construct the desired natural isomorphism on the generator ${}_AL^2A$ of $A\text{-}\Mod$:
\[
\Big({_B}\big(F({}_AL^2A)\big)\boxtimes_A-\Big)(L^2A)=
F({}_AL^2A)\boxtimes_AL^2A = F({}_AL^2A).
\]
This is a morphism in $B\text{-}\Mod$, and it is visibly $\End({}_AL^2A)$-equivariant.
\end{proof}

\begin{defn}\label{def: orthogonal direct sum of WCat}
Given a collection $\{C_i\}_{i\in I}$ of $\rmW^*$-categories, their \emph{orthogonal direct sum} $\boxplus C_i$ has objects formal sums $\boxplus_{i\in I} c_i$ with $c_i\in C_i$, and morphisms $\boxplus c_i\to \boxplus c'_i$ bounded families of morphisms $c_i \to c'_i$ in $C_i$.
\end{defn}

\begin{ex}
If $\{A_i\}$ is a collection of von Neumann algebras indexed by some set $I$,
then $\boxplus (A_i\text{-}\Mod) = (\prod A_i)\text{-}\Mod$.
\end{ex}

\begin{lem}
\label{lem: boxplus decomposition}
Let $C$ be a Cauchy complete $\rmW^*$-category, and let $c_i\in C$ be objects satisfying $\Hom(c_i,c_j)=0\,$ $\forall i\neq j$. Then $C$ splits as an orthogonal direct sum $C=(\boxplus_i C_i) \boxplus D$, where $C_i\subset C$ is the full subcategory generated by $c_i$, and $D=\{x\in C\,|\, \Hom(c_i,x)=0\;\forall i\}$.
\end{lem}

\begin{proof}
By Proposition~\ref{prop:freydembedding}, we may assume $C=A$-$\Mod$. Let $H_i$ be the underling Hilbert space of the object $c_i$, and let $p_i\in Z(A)$ be its central support, so that $p_iA$ is the image of $A$ in $B(H_i)$. Since $\Hom_A(H_i,H_j)=0$, the projections $p_i$ are orthogonal. Let $q:=1-\sum p_i$.
Then $A=(\prod_i p_iA)\oplus qA$.

Letting $C_i:=p_iA$-$\Mod$, and $D:=qA$-$\Mod$, we then have 
$A$-$\Mod\cong (\boxplus_i C_i) \boxplus D$.
Finally, since $H_i$ generates $C_i$,
the only objects that satisfy $\Hom(c_i,x)=0\;\forall i$ are those in $D$.
%
\end{proof}

If $D$ is a Cauchy complete full subcategory of some Cauchy complete $\rmW^*$-category $C$, then, by the previous lemma, $C$ decomposes as an orthogonal direct sum $C=D\boxplus D^\bot$, where $D^\bot=\{x\in C\,|\, \Hom(d,x)=0\;\forall d\in D\}$.
If we write
\begin{equation}
\label{eq: two orthogonal projections}
P:C\twoheadrightarrow D\qquad\text{and}\qquad
Q:C\twoheadrightarrow D^\bot
\end{equation}
for the two projections,
then every object $c\in C$ decomposes as a direct sum $c \cong P(c)\oplus Q(c)$
of its projections in $D$ and $D^\bot$.

\begin{lem}
Let $C=\boxplus_{i\in I} C_i$ and $D=\boxplus_{j\in J} D_j$ be $\rmW^*$-categories,
where the $D_j$ admit all direct sums.
Then the functor
\[
\Phi\,:\,\,\mathrm{Func}
\big(\boxplus_{i\in I} C_i,\boxplus_{j\in J} D_j\big)
\,\longrightarrow\,
\boxplus_{i\in I, j\in J}
\mathrm{Func}\big(C_i, D_j\big)\vspace{-2mm}
\]
which sends $F\in\mathrm{Func}(C,D)$ to the collection of all composites 
$C_i\hookrightarrow C\stackrel F\to D\twoheadrightarrow D_j$
is an equivalence of categories.
The inverse functor $\Psi:\boxplus_{i,j}
\mathrm{Func}(C_i, D_j)
\to
\mathrm{Func}(C,D)
$ sends $\boxplus_{i,j} F_{ij}$ 
to the functor $C\to D: 
\boxplus_i c_i\mapsto \boxplus_j (\oplus_i F_{ij}c_i)$.
\end{lem}

\begin{proof}
Clearly, $\Phi\circ\Psi\simeq\id$.
To see that $\Psi\circ\Phi\simeq\id$,
we construct an equivalence $F(\boxplus_i c_i)\cong(\Psi\Phi(F))(\boxplus_i c_i)$
for every $F\in \mathrm{Func}(C,D)$ and $\boxplus_i c_i\in C$, natural in $F$ and the $c_i$. Identifying $c_i$ with its image in $C$ and writing $[d]_j$ for the $j$th component of an object $d\in D$, the equivalence is given by
\(
(\Psi\Phi(F))(\boxplus_i c_i)=\boxplus_j (\oplus_i [F(c_i)]_j)=\oplus_i \boxplus_j [F(c_i)]_j=\oplus_i F(c_i)=F(\oplus_i c_i)=F(\boxplus_i c_i).
\)
\end{proof}

A $\rmW^*$-category which admits all direct sums is canonically tensored over $\Hilb$.
The action
\begin{equation}
\label{eq: action of Hilb}
\otimes:\Hilb \times C \to C
\end{equation}
of $\Hilb$ on such a $\rmW^*$-category is defined by Proposition~\ref{prop: direct sum completion}
as the unique up to contractible choice extension
of the trivial tensor functor $\mathbf{B}\bbC \to \mathrm{Func}(C,C): \star_\bbC \mapsto \id_C$ to a tensor functor $(\mathbf{B}\bbC)^{\bar \oplus} =\Hilb\to \mathrm{Func}(C,C)$.

Given $\rmW^*$-categories $C$ and $D$,
their \emph{tensor product} $C\bar\otimes D$\label{pageref tens prod of WCat} is the category with objects $c\otimes d$, for $c\in C$ and $d\in D$, and morphisms
\[
\Hom(c\otimes d,c'\otimes d') := p_2 \big(\End(c\oplus c') \,\bar \otimes\, 
\End(d\oplus d')\big) p_1,
\]
where $p_1=p_c\otimes p_d$ and $p_2=p_{c'}\otimes p_{d'}$ (and $p_x$ denotes the projection onto $x$, for $x\in \{c,d,c',d'\}$).
$\mathbf{B}\bbC$ is the unit of the above operation.

\begin{defn}\label{def: completed tensor product}
Given Cauchy complete $\rmW^*$-categories $C$ and $D$,
their \emph{completed tensor product} is given by:
\begin{equation}
\label{eq: completed tens of W*cat}
C \hat \otimes D := (C \,\bar\otimes\, D)^{\hat \oplus}.
\end{equation}
$\Hilb=(\mathbf{B}\bbC)^{\hat \oplus}$ is the unit of the above operation.
\end{defn}

The completed tensor product admits a canonical bilinear functor
$C \times D \to C \hat \otimes D$
given by $(c,d)\mapsto c\otimes d$.
And the canonical action \eqref{eq: action of Hilb} of $\Hilb$ on a Cauchy complete $\rmW^*$-category $C$ can be recovered as the unit isomorphism
$\Hilb\,\hat\otimes\,C\to C$
of the monoidal structure \eqref{eq: completed tens of W*cat},
pre-composed by the canonical bilinear functor
$\Hilb \times C\to \Hilb \,\hat\otimes\, C$.

The orthogonal direct sum of $\rmW^*$-categories distributes over the completed tensor product in the sense that there is a canonical equivalence of categories
\[
\big(\boxplus_{i\in I} C_i\big)\hat \otimes \big(\boxplus_{j\in J} D_j\big)
\,\,\,\cong\,\,\,
\boxplus_{\substack{i\in I \\ j\in J}}
(C_i\hat \otimes D_j).
\]

\begin{lem}
\label{lem:TensoratorForVNA->W*Cat}
Given von Neumann algebras $A$ and $B$, the functor 
\begin{align}
\label{eq: Otimes equivalence}
(A\text{-}\Mod)\,\hat\otimes\, (B\text{-}\Mod)
&\to\,\,
(A\bar\otimes B)\text{-}\Mod
\end{align}
that sends $({}_AH)\otimes ({}_BK)$ 
to ${}_{A\bar\otimes B}(H \otimes K)$
is an equivalence of categories.
\end{lem}
\begin{proof}
The functor 
\eqref{eq: Otimes equivalence} maps $({}_AL^2A)\otimes ({}_BL^2B)$ to ${}_{A\bar\otimes B}(L^2(A\bar\otimes B))$. 
Since both objects are generators, and since
$\End(({}_AL^2A)\otimes ({}_BL^2B))=\End({}_{A\bar\otimes B}(L^2(A\bar\otimes B)))$,
\eqref{eq: Otimes equivalence} is an equivalence by Corollary~\ref{cor: equivalence checkable on generators}.
\end{proof}

For $C$ and $D$ Cauchy complete $\rmW^*$-categories, there is a canonical inclusion
\begin{equation}
\label{eq: CxD versus Hom(Cbar,D)}
C \hat \otimes D \hookrightarrow \mathrm{Func}(\overline C, D),
\end{equation}
defined by Proposition~\ref{prop: idempotent completion} as the essentially unique extension of 
$
C \bar \otimes D \to \mathrm{Func}(\overline C, D):
c\otimes d\mapsto( \langle -, c\rangle \otimes d)$
along the functor 
$C \bar \otimes D\to C \hat \otimes D$.
The inclusion \eqref{eq: CxD versus Hom(Cbar,D)} is fully faithful, but rarely essentially surjective.

\begin{ex}
If $A$ and $B$ are von Neumann algebras, then 
\[
(A\text{-}\Mod)\,\hat\otimes\, (B\text{-}\Mod)=(A\bar\otimes B)\text{-}\Mod.
\]
The latter is a full subcategory of
\[
\mathrm{Func}\big(
\overline{A\text{-}\Mod}, B\text{-}\Mod\big)
=
\mathrm{Bim}(B, \overline A).
\]
\end{ex}

Generalising~\eqref{eq: action of Hilb}, if $C$ is a Cauchy complete $\rmW^*$-category, $c\in C$ is an object equipped with an action $A\to \End(c)$ of some von Neumann algebra $A$, and $H_A$ is a right $A$-module, then we can define a new object
\begin{equation}
\label{eq: Connes fusion with an object}
H \boxtimes_A c \,\in\, C
\end{equation}
as follows.
It is the value at $H_A\in A^{\op}\text{-}\Mod$ of the functor $-\boxtimes_A c:A^{\op}\text{-}\Mod\to\Hilb$ uniquely specified, by Corollary~\ref{cor: enough to construct on generator}, by the requirement that $L^2A \mapsto c$.
Similarly, if $f:c_1\to c_2$ is an $A$-equivariant morphism between $A$-objects of $C$, then $H \boxtimes_A f:H \boxtimes_A c_1\to H \boxtimes_A c_2$ is the value at $H_A$ of the natural transformation  $A^{\op}\text{-}\Mod\tworarrow C$ 
uniquely specified by the requirement that $L^2A \mapsto f$.
The special case $C=\Hilb$ of the above construction recovers the classical operation of Connes fusion, while the special case $A=\bbC$ recovers the action \eqref{eq: action of Hilb} of $\Hilb$ on $C$.
\begin{lem}
The operation \eqref{eq: Connes fusion with an object} comes equipped with associator and unitor natural transformations
\[
(H \boxtimes_A K) \boxtimes_B c \stackrel{\cong}\to H \boxtimes_A (K \boxtimes_B c)
\qquad\qquad
L^2A\boxtimes_A c \stackrel{\cong}\to c
\]
that satisfy the usual pentagon and triangle axioms.
\end{lem}
\begin{proof}
By definition, $-\boxtimes_Ac$ maps $L^2A_A$ to $c$, so the unitors are just identity morphisms.
To define the associator,
it suffices by Corollary~\ref{cor: enough to construct on generator} to construct it for $H=L^2A$, in which case it is given by the following composition of unitors:
\[
(L^2A \boxtimes_A K) \boxtimes_B c \to K \boxtimes_B c \to L^2A \boxtimes_A (K \boxtimes_B c).
\]
The triangle axiom is easy, and left as an exercise to the reader.
To check the pentagon axiom for the various parenthesizations of $(-\boxtimes_{A_1}-\boxtimes_{A_2}-\boxtimes_{A_3}c)$, note that the action of $A_3$ on $c\in C$ induces a functor $F:\mathbf BA_3^{\hat\oplus}\cong A_3^{\op}\text{-}\Mod\to C$, and that the pentagon diagram for $c$ is the image under $F$ of the pentagon diagram for $L^2A_3\in A_3^{\op}\text{-}\Mod$. The latter is an instance of the usual pentagon diagram for Connes fusion, which is well known to commute (see e.g. \cite{10.1063/1.1563733}).
Its image under $F$ therefore also commutes.
\end{proof}

\subsection*{A technical von Neumann algebra lemma}

The following lemma is used in the proof of Proposition~\ref{prop: direct sum completion}.
Recall that $J\Subset I$ means that $J\subset I$ and $|J|<\infty$.

\begin{lem}
\label{lem: homomorphism defined on union of pMp}
Let $M$, $N$ be von Neumann algebras, 
and let $\{p_i\}_{i\in I}$ be a family of projections in $M$ that add up to $1$. 
For every 
$J\Subset I$,
let 
$p_J:=\sum_{i\in J}p_i$.
If 
\[
\varphi: 
\bigcup_{J\Subset I} p_JMp_J
\to N
\]
is a $*$-homomorphism whose restriction to each $p_JMp_J$
is normal, and such that $\sum_i \varphi(p_i)= 1$ in $N$, then $\varphi$ extends uniquely to a homomorphism of von Neumann algebras $M\to N$.
\end{lem}
\begin{proof}
Let $M_0:=\bigcup_{J\Subset I}p_JM p_J$,
and note that $p_JM p_J=p_J(M_0) p_J$ for all $J\Subset I$.
By realising $N$ on some Hilbert space $H$, we may assume without loss of generality that $N=B(H)$.
Let $A\subset B(L^2M\oplus H)$ be the von Neumann algebra generated by the image of $\lambda\oplus\varphi$, where $\lambda$ denotes the left action of $M$ on $L^2(M)$.
Let $q_i:=(\lambda\oplus\varphi)(p_i)$,
let $q_J:=\sum_{i\in J} q_i$, and
note that $\sum q_i=1$ in $A$.
Consider the following commutative diagram:
\[
\begin{tikzcd}
& M_0 \arrow[d, "\jmath"'] \arrow[dr, "\varphi"] \arrow[dl, "\iota"'] &\\ M \arrow[<<-, r] & A & N \arrow[<-, l]
\end{tikzcd}
\]
where $\iota:M_0 \to M$ is the inclusion, and $\jmath:M_0 \to A$ is induced by $\lambda\oplus\varphi$.
The map $\jmath$ has ultrawekly dense image.
Its compression
\begin{equation}\label{eq: ksdfbvhjkd}
\jmath\,:\,p_J(M_0)p_J \rightarrow q_JAq_J
\end{equation}
therefore also has dense image.
Since
\begin{equation}\label{eq: ksdfbvhjkd BIS}
p_JMp_J=p_J(M_0)p_J \,\to\, q_JAq_J \,\twoheadrightarrow\, p_JMp_J
\end{equation}
is the identity map, \eqref{eq: ksdfbvhjkd} is also injective.
A homomorphism between von Neuman algebras which is
both injective and with dense image is necessarily an isomorphism.
So \eqref{eq: ksdfbvhjkd} is an isomorphism, and all maps in \eqref{eq: ksdfbvhjkd BIS} are in fact isomorphisms.

Let $K:=\ker(A\to M)$,
so that $A = A_0 \oplus K$ by Lemma~\ref{lem: kernel of maps of vNalg}.
For each $J\Subset I$, the projection $q_J \in A$ does not have a $K$ component, as otherwise the map $q_JAq_J \to p_JMp_J$ would not be an isomorphism.
So $1=\sup q_J$ has no $K$ component, and $K=0$.
So the map $A\to M$ is an isomorphism.

The desired homomorphism $M\to N$ is then the composite of the map $A \to N$ with the inverse of the isomorphism $A\to M$.
\end{proof}

\section{The Hilb-valued inner product}

Cauchy complete $\rmW^*$-categories behave in many ways like Hilbert spaces. 
For example, the action \eqref{eq: action of Hilb} is the analog of the fact that any Hilbert space $H$ admits the structure of a $\bbC$-module.
$\rmW^*$-categories also have an inner product:

\begin{defn}
\label{def: Hilb valued inner product}
Every\, $\rmW^*$-category\, $C$\, admits\, a\, canonical \emph{$\Hilb$-valued\, inner\, product}\\
$\langle-,-\rangle_\Hilb:\, \overline{C}\times C\to \Hilb$
given by:
\begin{equation}
\label{eq: statement: Hilb-valued inner product}
    \langle a,b\rangle_\Hilb \,:=\,
    p_b L^2( \mathrm{End}(a\oplus b))p_a,
\end{equation}
where $p_a, p_b \in\End(a\oplus b)$ are the two projections.

The Hilb-valued inner product is compatible with direct sums in the sense that there are canonical unitary isomorphisms
$\langle a,\oplus_i b_i\rangle
\cong
\oplus_i \langle a,b_i\rangle
$ and $
\langle \oplus_i a_i, b\rangle
\cong
\oplus_i \langle a_i,b\rangle$.

Morphisms $f:b_1\to b_2$ and $g:a_1\to a_2$ in $W$ induce maps
\[
f_*:\langle a,b_1\rangle \to \langle a,b_2\rangle
\qquad\qquad
g_*:\langle a_1,b\rangle \to \langle a_2,b\rangle
\]
defined as the restriction to the appropriate corner of the left action
of $\left(\begin{smallmatrix}
0 & 0 \\ f & 0
\end{smallmatrix}\right)\in\End(b_1\oplus b_2)$ on $\langle a,b_1\oplus b_2\rangle \cong$ ${\langle a,b_1\rangle \oplus \langle a,b_2\rangle}$,
and 
of the right action of $\left(\begin{smallmatrix}
0 & g^* \\ 0 & 0
\end{smallmatrix}\right) \in \End(a_1\oplus a_2)$ on 
$\langle a_1\oplus a_2, b\rangle\cong \langle a_1,b\rangle \oplus \langle a_2,b\rangle$.

The modular conjugation $J$ on 
$L^2\left( \mathrm{End}\left(x\oplus y\right)\right)$
induces a unitary isomorphism
\[
J_{x,y}\colon \langle x,y\rangle_{\Hilb}
\to\overline{\langle y,x\rangle_{\Hilb}}.
\]
It is involutive in the sense that
$\overline{J_{y,x}}\circ J_{x,y} = \id_{\langle x,y\rangle}$,
and $J_{x,x}$ agrees with the modular conjugation on
\begin{equation}
\label{eq: <X,X> is L^2}
\langle x,x\rangle_{\Hilb}=L^2(\mathrm{End}(x)).
\end{equation}
\end{defn}

\begin{lem}
$\langle x,y\rangle_{\Hilb}=0$ if and only if $\Hom(x,y)=0$.
\end{lem}
\begin{proof}
Let $A:=\mathrm{End}(x\oplus y)$.
Then $\langle x,y\rangle_{\Hilb}=p_y L^2Ap_x$ while
$\Hom(x,y)=p_yAp_x$.
The conditions $p_yAp_x=0$ and $p_yL^2Ap_x=0$ are both equivalent to $p_x$ and $p_y$ having disjoint central support.
\end{proof}

\begin{lem}
Let $f:y_1\to y_2$ be a morphism in a $\rmW^*$-category.
Then the two maps $\langle x,y_2\rangle \to \langle x,y_1\rangle$
provided by $(f_*)^*$ and $(f^*)_*$ agree.

Similarly, for $g:x_1\to x_2$, the two maps $(g_*)^*,(g^*)_*:\langle x_2,y\rangle \to \langle x_1,y\rangle$ agree.
\end{lem}

\begin{proof}
We prove the first statement.
Let $A:=\End(x\oplus y_1\oplus y_2)$, so that
$\langle x,y_1\rangle = p_{y_1}(L^2A)p_x$ and
$\langle x,y_2\rangle = p_{y_2}(L^2A)p_x$, 
and let $\ell:A\to B(L^2A)$ be the left action of $A$ on its $L^2$ space. 
Using that $\ell$ is a $*$-algebra homomorphism, we compute:
\[
(f_*)^*=\ell(\left(\begin{smallmatrix}
0 & 0 & 0 \\0 & 0 & 0 \\ 0 & f & 0
\end{smallmatrix}\right))^*=\ell(\left(\begin{smallmatrix}
0 & 0 & 0 \\0 & 0 & 0 \\ 0 & f & 0
\end{smallmatrix}\right)^*)=(f^*)_*. \qedhere
\]
\end{proof}

\begin{lem}
Let $x$ and $y$ be objects of some $\rmW^*$-category. Then the images of $\overline{\End(x)}$ and $\End(y)$ in $B(\langle x,y\rangle)$ are each other's commutants.
\end{lem}

\begin{proof}
Let $z_x,z_y\in Z(\End(x\oplus y))$ be the central supports of $p_x$ and $p_y$ (where $p_x, p_x \in\End(x\oplus y)$ are the two projections), and let $z:=z_pz_q$.
Letting $A=z\End(x\oplus y)$,
$p=zp_x$, and $q=zp_y$, we then have $\langle x,y\rangle=qL^2Ap$.
Since $p$ and $q$ have full central support in $A$, the Hilbert space
$qL^2Ap$ is a Morita equivalence between $qAq$ and $pAp$.

The result follows since
the image of $\End(y)$ in $B(qL^2Ap)$ is $qAq$, and the image of $\overline{\End(x)}=\End(X)^{\op}$ is $pA^{\op}p$.
\end{proof}

\begin{defn}\label{def: orthog gens}
A set $\{c_i\}_{i\in I}$ of generators of some $\rmW^*$-category is said to be orthogonal if $\langle c_i,c_j\rangle =0$ $\forall i\neq j$. Equivalently, if $\Hom(c_i,c_j) =0$ $\forall i\neq j$.
\end{defn}

The classical Gram–Schmidt process admits a variant for Cauchy complete $\rmW^*$-categories:

\begin{prop} [Gram–Schmidt for $\rmW^*$-categories]
\label{prop: Gram–Schmidt}
Let $C$ be a Cauchy complete $\rmW^*$-category, and
let $\{c_i\}_{i\in I}$ be set of generators indexed by some well ordered set $I$.
For each $i\in I$, let $C_{<i}$ be the full subcategory generated by $\{c_j\}_{j<i}$, and let $C_{\le i}$ be the full subcategory generated by $\{c_j\}_{j\le i}$.
Let $C_i$ be the orthogonal complement of $C_{<i}$ inside $C_{\le i}$,
and let
\begin{equation*}
P_i:C_{\le i}\twoheadrightarrow C_{<i}\qquad\text{and}\qquad
Q_i:C\twoheadrightarrow C_i
\end{equation*}
be the two orthogonal projections, as in \eqref{eq: two orthogonal projections}.
Then $C=\boxplus_{i}C_i$, $Q_i(c_i)$ generates $C_i$, and $\{Q_i(c_i)\}_{i\in I}$ is an orthogonal set of generators for $C$.
\end{prop}

\begin{proof}
Identify $I$ with the set of ordinals strictly smaller than some given ordinal $\alpha_0$.
We first prove, by transfinite induction, that for all $i\le \alpha_0$ we have $C_{<i}=\boxplus_{j<i}C_j$:
\begin{itemize}
\item (base case) $C_{<0}=\{0\}$ is the orthogonal direct sum indexed by the empty set $\emptyset$.
\item (successor case) Assume
$C_{<i}=\boxplus_{j<i}C_j$. By construction, $C_{< i+1}=C_{\le i}=C_{< i}\boxplus C_i$. 
So $C_{<i+1}=\boxplus_{j<i+1}C_j$.
\item (limit case) Let $i$ be a limit ordinal,
and let us assume that
$C_{<j}=\boxplus_{k<j}C_k$ for all $j<i$.
By construction, we then have $C_{< i}=(\bigcup_{j<i} C_{< j})^{\hat\oplus}=
(\bigcup_{j<i} (\boxplus_{k<j}C_k))^{\hat\oplus}
=\boxplus_{j<i}C_j$.
\end{itemize}
In particular, we have $C=\boxplus_{i}C_i$.

Let $a_i=P_i(c_i)$ and $b_i=Q_i(c_i)$, so that $a_i\oplus b_i=c_i$ as in \eqref{eq: two orthogonal projections}.
Since $b_i$ is in the orthogonal complement of $C_{<i}$, it is orthogonal to all the $b_j$ with $j<i$. So $\{b_i\}_{i\in I}$ is an orthogonal set.
To see that $b_i$ generates $C_i$,
choose a non-zero object $x\in C_i$.
Since $x\in C_{\le i}$, it receives a non-zero map from $c_j$ for some $j\le i$.
If $j<i$, then $\Hom(c_j,x)=0$ because $c_j\in C_{<i}$, so the only possible source of a non-zero map to $x$ is $c_i$.
So $\Hom(c_i,x)=\Hom(b_i,x)\neq 0$, and
$b_i$ generates $C_i$. It follows that $\{b_i\}_{i\in I}$ generates $C=\boxplus_{i}C_i$.
\end{proof}

\begin{lem}
\label{lem: antiunitaries at level of <,>}
If $F:C\to D$ is an equivalence of $\rmW^*$-categories,
then we have canonical unitaries
$\langle c_1,c_2\rangle \to \langle F(c_1),F(c_2)\rangle$, natural in
$c_1$ and $c_2$.
Similarly, if $F:C\to D$ is an antilinear equivalence (i.e. an equivalence $\overline C\to D$)
then we have canonical antiunitaries
$\langle c_1,c_2\rangle \to \langle F(c_1),F(c_2)\rangle$, again natural in
$c_1$ and $c_2$.

These are functorial in the sense that whenever $F$ and $G$ are composable functors (either linear or antilinear), the following diagram commutes:
\[
\begin{tikzpicture}
\node (A) at (0,1.2)
{$\langle G(c_1),G(c_2)\rangle$};
\node (B) at (-3,0)
{$\langle c_1,c_2\rangle$};
\node (C) at (3,0)
{$\langle FG(c_1),FG(c_2)\rangle.$};
\draw[->] (B) -- (A);
\draw[->] (B) -- (C);
\draw[->] (A) -- (C);
\end{tikzpicture}
\]
\end{lem}

\begin{proof}
All the constructions in the Definition~\ref{def: Hilb valued inner product}
are natural with respect to $\bbR$-linear isomorphisms.
So all we need to do is check that the 
isomorphism $\langle c_1,c_2\rangle \to \langle F(c_1),F(c_2)\rangle$ associated to an antilinear functor $F$ is genuinely antilinear.
It is enough to check this for the antilinear identity functor from $\overline C$ to $C$:
\begin{align*}
\langle \overline{x},\overline{y}\rangle =
p_{\overline{y}} L^2( &\mathrm{End}(\overline{x}\oplus \overline{y}))p_{\overline{x}}
=
p_{\overline{y}} L^2( \overline{\mathrm{End}(x\oplus y)})p_{\overline{x}}\\
&=
p_{\overline{y}} \overline{L^2( \mathrm{End}(x\oplus y))}p_{\overline{x}}
=
\overline{p_y L^2( \mathrm{End}(x\oplus y))p_x}
=
\overline{\langle x,y\rangle}.
\qedhere
\end{align*}
\end{proof}

Our next goal is to refine Proposition~\ref{prop:freydembedding} by providing an explicit formula for the equivalence whose existence is guaranteed by the lemma.
We will similarly refine Corollary~\ref{cor: enough to construct on generator} by providing explicit formulas for the extensions whose existence and uniqueness are guaranteed by the corollary.

\begin{prop}
\label{prop:freydembedding'}
If $C$ is a Cauchy complete $\rmW^*$-category,
$c\in C$ is a generator, and $A:=\overline{\End(c)}$, then
\begin{align}
\notag
F:C&\to A\text{-}\Mod
\\
\label{eq: freydembedding'}
x&\mapsto\;\! \langle c,x\rangle
\end{align}
is an equivalence of categories.
\end{prop}
\begin{proof}
The map\\
$
F:\End(c) \to \End_{A\text{-}\Mod}(F(c))=
\End_{A\text{-}\Mod}(\langle c,c\rangle)=
\End(L^2\End(c)_{\End(c)})=\End(c)
$
is an isomorphism (it's the identity map).
So $F$ is an equivalence by
Corollary~\ref{cor: equivalence checkable on generators}.
\end{proof}

Let $c\in C$, $d\in D$, $\rho:\End(c)\to \End(d)$ be as in the first part of Corollary~\ref{cor: enough to construct on generator},
and let $F:c\mapsto d$ be the functor $\{c\}\to D$ given by $\rho$.
Then the unique up to contractible choice extension of $F$ to the whole of $C$ is given by
\begin{align}
\notag
C\,&\longrightarrow\,\,\,\, D\\
\label{eq: formula for unique extension}
x \,&\,\mapsto\,\, \langle c,x \rangle\boxtimes_{\End(c)} d.
\end{align}
Indeed, setting $x=c$, equation \eqref{eq: formula for unique extension} recovers $\langle c,c \rangle\boxtimes_{\End(c)} F(c)=L^2\End(c)\boxtimes_{\End(c)} F(c)=F(c)$, as desired.
A variant of \eqref{eq: formula for unique extension}
is presented in \eqref{eq: witty} below.
Similarly, if $f:d_1\to d_2\in D$ is as in the second part of Corollary~\ref{cor: enough to construct on generator}, 
then the unique natural transformation $C\tworarrow D$ extending  $\{c\}\tworarrow D:c\mapsto f$ is given by
\begin{equation}
\label{eq: formula for unique extension 2}
x \,\,\mapsto \,\, \big(\id\boxtimes f: \langle c,x \rangle\boxtimes_{\End(c)} d_1 \to \langle c,x \rangle\boxtimes_{\End(c)} d_2\big).
\end{equation}

We pause to emphasise the similarities between Hilbert spaces and $\rmW^*$-categories:
\begin{rem}\label{rem: witty}
If $H$ and $K$ are Hilbert spaces, and $\{e_i\}$ is an orthogonal basis of $H$, then a map $f:H\to K$ is fully determined by the images $f(e_i)$ of the basis vectors, via the well-known formula
\begin{equation}\label{eq: witty 1}
f(\xi) = \sum_i\frac{\langle e_i,\xi\rangle f(e_i)}{\|e_i\|^2}.
\end{equation}
If $H$ and $K$ are $\rmW^*$-categories, and $\{e_i\}$ is a set of generators of $H$ that are orthogonal in the sense of Definition~\ref{def: orthog gens}, 
then a functor $f:H\to K$ is fully determined by the objects $f(e_i)$, via the formula
\begin{equation}
\label{eq: witty}
f(\xi) = \bigoplus_i \underset{\End(e_i)}{\langle e_i,\xi \rangle\,\,\boxtimes\,\,\, f(e_i).}
\end{equation}
\end{rem}

Given the importance of \eqref{eq: formula for unique extension}, we record it in yet another form:

\begin{lem}
Let $F:C\to D$ be a functor between $\rmW^*$-categories.
Let $c\in C$ be a generator, and let $x\in C$ be an arbitrary object. Then there is a canonical unitary isomorphism
\begin{equation}
\label{eq: <.,.> boxtimes F(X) = F(Y)}
\langle c,x\rangle \boxtimes_{\End(c)}F(c) \cong F(x),
\end{equation}
natural in $x$.
\end{lem}
\begin{proof}
By Corollary \ref{cor: enough to construct on generator}, it is enough to construct the isomorphism \eqref{eq: <.,.> boxtimes F(X) = F(Y)} for $x=c$. Recall from \eqref{eq: <X,X> is L^2} that $\langle c,c\rangle=L^2(\mathrm{End}(c))$. The desired isomprhism $L^2(\End(c)) \boxtimes_{\End(c)}F(c) \cong F(c)$ is then just a unitor,
and the naturality requirement is the fact that the unitor commutes with the left actions of $\End(c)$.
\end{proof}

\begin{lem}
\label{lem: formula for Hilb valued inner product}
If $C=A\text{-}\Mod$ for some von Neumann algebra $A$, then the $\Hilb$-valued inner product \eqref{eq: statement: Hilb-valued inner product} is given by
\begin{equation} \label{eq: <X,Y> as boxtimes}
\langle X,Y\rangle_\Hilb=\overline X\boxtimes_AY.
\end{equation}
\end{lem}

\begin{proof}
Since $A\text{-}\Mod$ is the Cauchy completion of its full subcategory $\{{}_AL^2A\}\cong \mathbf{B}A^{\op}$,
to construct a natural isomorphism 
between $(X,Y)\mapsto \langle X,Y\rangle_\Hilb$ and 
$(X,Y)\mapsto \overline X\boxtimes_AY$,
it is enough, by Corollary~\ref{cor: Cauchy completion'}, to specify it for $X=Y={}_AL^2A$.
So we only need to construct an isomorphism
\begin{equation}\label{eq: <L^2,L^2>}\langle {}_AL^2A,{}_AL^2A\rangle_\Hilb=\overline {L^2A}\boxtimes_A L^2A,
\end{equation}
compatible with the two actions of $\End({}_AL^2A)=A^{\op}$.
We now observe that both sides of \eqref{eq: <L^2,L^2>} are canonically isomorphic to $L^2A$, compatibly with the actions of $A$.
\end{proof}

Recall that all our $\rmW^*$-categories are assumed to admit a set of generators.
The first part of the next proposition is an analog of the Riesz representation theorem.
The second part of the proposition is reminiscent of the Yoneda lemma, and is a categorial analog of Lemma~\ref{lem: pre-Hilbert space}:

\begin{prop}
\label{prop:riesz}
(i) Let $C$ be a Cauchy complete $\rmW^*$-category.
Then 
\begin{align}
\label{eq: Riesz map}
C &\to \mathrm{Func}(\overline C,\Hilb)\\\notag
c &\mapsto\,\, \left\langle-,c\right\rangle
\end{align}
is an equivalence of categories.

(ii) If $C$ is an arbitrary $\rmW^*$-category, then $\mathrm{Func}(\overline C,\Hilb)\cong C^{\hat\oplus}$, and \eqref{eq: Riesz map} is the canonical functor
from $C$ to its Cauchy completion.
\end{prop}

\begin{proof}
{\it(i)} Let $x\in C$ be a generator,
let $A:=\End(x)$,
and let $\mathbf{B}A\hookrightarrow C$ be the full subcategory on $x$,
so that $(\mathbf{B}A)^{\hat \oplus}\cong  C$.
The restriction functor
$\mathrm{Func}(\overline C,\Hilb)\to
\mathrm{Func}(\overline{\mathbf{B}A},\Hilb)=
\overline {A}\text{-}\Mod
$
is an equivalence
by Corollary~\ref{cor: Cauchy completion}, 
while the composite
\[
C \stackrel{\eqref{eq: Riesz map}}\longrightarrow
\mathrm{Func}(\overline C,\Hilb)\to
\mathrm{Func}(\overline{\mathbf{B}A},\Hilb)=
\overline {A}\text{-}\Mod
\]
is an equivalence by Proposition~\ref{prop:freydembedding'}.
So \eqref{eq: Riesz map} is an equivalence.

{\it (ii)} The functor \eqref{eq: Riesz map} factors as
$
C\longrightarrow C^{\hat\oplus} \xrightarrow{\simeq} \mathrm{Func}(\overline{C^{\hat\oplus}},\Hilb) \xrightarrow{\simeq} \mathrm{Func}(\overline C,\Hilb)$.
The middle arrow is an equivalence by part {\it (i)} of the Proposition, and the last one is an equivalence by Corollary~\ref{cor: Cauchy completion}.
\end{proof}

\begin{rem}
The above proposition generalises the `unitary Yoneda lemma'
(used in \cite{MR3687214,MR3948170,MR4079745,MR4598730} and made explicit in \cite[\S 2.1]{MR4750417}) from semisimple $\rmC^*$-categories to arbitrary $\rmW^*$-categories.
Specifically, if $C$ is a semisimple $\rmC^*$-category equipped with a unitary trace (\cite[Def.~2.1]{MR4750417}) then there is a canonical isomorphism $\Hom(X,Y)\cong \langle X,Y\rangle$ induced by the trace\footnote{Different traces induce different isomorphisms $\Hom(X,Y)\cong \langle X,Y\rangle$.}, and \eqref{eq: Riesz map} becomes the Yoneda embedding.
\end{rem}

\begin{defn} \label{def: adjoint functors}
Given a functor $F\colon C\to D$ between Cauchy complete $\rmW^*$-categories,\vspace{-1mm} its adjoint\footnote{See footnote~\ref{footnote 2}.} is defined as the composite
$
F^\dagger:D \xrightarrow{\simeq}
\mathrm{Hom}(\overline{D},\Hilb)
\xrightarrow{-\circ\overline F}
\mathrm{Hom}(\overline{C},\Hilb)
\xrightarrow{\simeq}
C
$.
Equivalently, the functor $F^\dagger:D\to C$ is specified by the requirement that
\begin{equation}
\label{eq: defining property of adjoint}
\langle c,F^\dagger(d)\rangle_\Hilb
\,\,\cong\,\,
\langle F(c),d\rangle_\Hilb
,
\end{equation}
naturally in $c$ and $d$.

Given a natural transformation $\alpha:F\Rightarrow G$ between functors $C\to D$, the adjoint natural transformation $\alpha^\dagger:F^\dagger\Rightarrow G^\dagger$ is specified, using Proposition~\ref{prop:riesz}, by the requirement that the diagram
\[
\begin{tikzcd}
\big\langle c,F^\dagger(d)\big\rangle_\Hilb
\arrow[r,"\simeq"]\arrow[d, "\langle\,\,\,\,{,}(\alpha^\dagger)_{d}\rangle"']&
\big\langle F(c),d\big\rangle_\Hilb
\arrow[d, "\langle\alpha_{c}{,}\,\,\,\,\rangle"]
\\
\big\langle c,G^\dagger(d)\big\rangle_\Hilb
\arrow[r,"\simeq"]&
\big\langle G(c),d\big\rangle_\Hilb
\end{tikzcd}
\]
commutes. The operations $F\mapsto F^\dagger$ and $\alpha\mapsto \alpha^\dagger$ assemble to an antilinear equivalence
\begin{equation}
\label{eq: dag}
\dagger:\mathrm{Func}(C,D)\to \mathrm{Func}(D,C),
\end{equation}
and there are natural unitary isomorphisms 
$\varphi_F:F\to F^{\dagger\dagger}$ and
$\nu_{F,G}:F^\dag\circ G^\dag \to (G\circ F)^\dag$ defined as follows.

The first one is defined, using Proposition~\ref{prop:riesz}, by the requirement that $\langle (\varphi_F)_c,d\rangle$ be equal to the composite
$$
\langle F(c),d\rangle
\cong
\langle c,F^\dag(d)\rangle
\xrightarrow{J}
\overline{
\langle F^\dag(d),c\rangle
}
\cong
\overline{
\langle d,F^{\dag\dag}(c)\rangle
}
\xrightarrow{J^{-1}}
\langle F^{\dag\dag}(c), d\rangle.
$$
The second is defined similarly, by asking that $\langle c,(\nu_{F,G})_d\rangle$ be equal to the composite
$$
\langle c, F^\dag G^\dag(d)\rangle
\xrightarrow{\simeq}
\langle F(c), G^\dag(d)\rangle
\xrightarrow{\simeq}
\langle GF(c), d\rangle
\xrightarrow{\simeq}
\langle c, (G F)^\dag(d)\rangle.
$$
These isomorphisms $\varphi$ and $\nu$ satisfy coherences listed below, in Lemma~\ref{lem: coherences!}. 
\end{defn}

\begin{warn*}
The operation $\dagger$ defined above should not be confused with the antilinear involution $*:\mathrm{Func}(C,D)\to \mathrm{Func}(C,D)$ coming from the fact that $\mathrm{Func}(C,D)$ is a $*$-category.
\end{warn*}

\begin{rem}
When $C$ and $D$ are semisimple $\rmC^*$-categories equipped with unitary traces, one can identify the functor $F^\dagger:D\to C$ with the categorical adjoint of $F:C\to D$. In that case, \eqref{eq: defining property of adjoint} becomes a unitary adjunction in the sense of \cite[Def 2.3]{MR4750417}. 
Note however that the identification between $F^\dagger$ and the categorical adjoint of $F$ genuinely depends on the unitary traces of $C$ and $D$.
\end{rem}

\begin{lem}
\label{lem: conjugate bimodule}
If a functor $F:A\text{-}\Mod\to B\text{-}\Mod$
corresponds to the bimodule ${}_BX_A$ under the equivalence \eqref{eq: Func equivalence},
then its adjoint $F^\dagger:B\text{-}\Mod\to A\text{-}\Mod$
corresponds to the complex conjugate bimodule ${}_A\overline{X}_B$.
\end{lem}

\begin{proof}
Using \eqref{eq: <X,Y> as boxtimes}, the following sequence of natural isomorphisms
\[
\langle {}_BX\boxtimes_A H,{}_BK\rangle
=
\big(\overline{{}_BX\boxtimes_A H}\big) \boxtimes_B K
\cong
\overline H \boxtimes_A \overline X \boxtimes_B K
=
\langle {}_AH, {}_A \overline X \boxtimes_B K \rangle
\]
shows that ${}_A\overline{X}_B$ satisfies the defining property of $F^\dagger$.
\end{proof}

\begin{lem}
A functor $F:C\to D$ between Cauchy complete $\rmW^*$-categories is faithful if and only if its adjoint $F^\dagger:D\to C$ is dominant.
\end{lem}
\begin{proof}
Pick isomorphisms $C\cong A$-$\Mod$ and $D\cong B$-$\Mod$ by Proposition~\ref{prop:freydembedding}, and let ${}_BX_A$ be the bimodule that represents $F$ under the equivalence \eqref{eq: Func equivalence}.
The conjugate bimodule ${}_A\overline{X}_B$ then represents $F^\dagger$ by Lemma~\ref{lem: conjugate bimodule}.

Now note that $F:A$-$\Mod\to B$-$\Mod$ is faithful if and only if the action of $A$ on the corresponding bimodule ${}_BX_A$ is faithful, and is dominant if and only if the action of $B$ on ${}_BX_A$ is faithful.
\end{proof}

We finish this section by listing and proving the coherences satisfied by the isomorphisms $\varphi$ and $\nu$ from Definition~\ref{def: adjoint functors}:

\begin{lem}\label{lem: coherences!}
The isomorphisms $\varphi_F$ and $\nu_{F,G}$ satisfy:
\begin{itemize}
\item\vspace{-2mm}
$\varphi_{F^\dagger}=(\varphi_F)^\dagger$
\item\vspace{-2mm}
$\nu_{F,H\circ G}(\id_{F^\dag} \circ \nu_{G,H}) =
\nu_{G \circ F, H}(\nu_{F,G} \circ \id_{H^\dag})$
\item\vspace{-2mm}
$\varphi_{F \circ G}=(\nu_{G,F})^\dagger \nu_{F^\dagger,G^\dagger} (\varphi_F\circ\varphi_G)$
\end{itemize}
\end{lem}
\begin{proof}
The first equation holds, using Proposition~\ref{prop:riesz}, by the following commutative diagram:\vspace{-5mm}
\[
\begin{tikzpicture}
\node (B) at (3.9,-1.3) {$\overline{\langle F^{\dag}(d),c\rangle}$};
\node (B') at (8.1,-1.3) {$\overline{\langle F^{\dag\dag\dag}(d), c\rangle}$};
\node (A) at (0,0) {$\langle c,F^{\dag}(d)\rangle$};
\node (C) at (12,0) {$\langle c, F^{\dag\dag\dag}(d)\rangle$};
\node (E) at (6,0) {$\overline{\langle d,F^{\dag\dag}(c)\rangle}$};
\node (D) at (3.9,1.3) {$\langle F(c),d\rangle$};
\node (D') at (8.1,1.3) {$\langle F^{\dag\dag}(c), d\rangle$};
\draw[->] (A) -- (D);
\draw[->] (A) -- (B);
\draw[->] (B) --node[below]{$\scriptstyle \overline{\langle (\varphi_{F^\dag})_d,\,\,\rangle}$} (B');
\draw[->] (B') -- (C);
\draw[->] (B) -- (E);
\draw[->] (E) -- (D');
\draw[->] (D) -- node[above]{$\scriptstyle \langle(\varphi_F)_c,\,\,\rangle$} (D');
\draw[->] (D') -- (C);
\draw[->] (A) to[bend left=42] node[above]{$\scriptstyle \langle\,\,,(\varphi_F^\dag)_d\rangle$}
(C);
\draw[->] (A) to[bend right=42]node[below]{$\scriptstyle \langle \,\,,(\varphi_{F^\dag})_d\rangle$} (C);
\end{tikzpicture}\vspace{-4mm}
\]
The second equation holds by:
$$
\begin{tikzcd}
\langle c, F^\dag G^\dag H^\dag(d)\rangle
\arrow[dr]
\arrow[rrr,
"{\langle \,\,,(\nu_{F,G})_{H^\dag(d)}\rangle}"
]
\arrow[ddd,
"{\langle \,\,,F^\dag((\nu_{G,H})_{d})\rangle}"
description
]
&&&
\langle c, (G F)^\dag H^\dag(d)\rangle
\arrow[ddd,
"{\langle \,\,,(\nu_{G F,H})_{d}\rangle}" description
]
\arrow[-, dl]
\\[-3mm]
&
\langle F(c), G^\dag H^\dag(d)\rangle
\arrow[r]
\arrow[d, "{\langle \,\,,(\nu_{G,H})_{d}\rangle}",
swap
]
&
\langle GF(c), H^\dag(d)\rangle
\arrow[d]
\\
&
\arrow[r]
\langle F(c),(H G)^\dag (d)\rangle
&
\langle HGF(c), d\rangle
\\[-3mm]
\langle c, F^\dag (H G)^\dag (d)\rangle
\arrow[ur]
\arrow[rrr, 
"{\langle \,\,,(\nu_{F,H G})_{d}\rangle}"'
]
&&&
\langle c, (H G F)^\dag(d)\rangle
\arrow[<-, ul]
\end{tikzcd}\vspace{1mm}
$$
And the third one holds by:
\[
\begin{tikzpicture}[xscale=1.3]
\node (A) at (0,4) {$\langle FG(c),d \rangle$}; 
\node (B) at (3,4) {$\langle c,(FG)^\dag(d) \rangle$}; 
\node (C) at (5.75,4) {$\overline{\langle (FG)^\dag(d),c \rangle}$}; 
\node (D) at (8.5,4) {$\overline{\langle d,(FG)^{\dag\dag}(c) \rangle}$}; 
\node (E) at (11.25,4) {$\langle (FG)^{\dag\dag}(c),d \rangle$}; 
\node (F) at (1.25,2) {$\langle G(c),F^\dag(d) \rangle$}; 
\node (G) at (3.55,2) {$\langle c,G^\dag F^\dag(d) \rangle$}; 
\node (H) at (5.75,2) {$\overline{\langle G^\dag F^\dag(d),c \rangle}$}; 
\node (I) at (8.5,2) {$\overline{\langle d,(G^\dag F^\dag)^\dag(c) \rangle}$}; 
\node (J) at (11.25,2) {$\langle (G^\dag F^\dag)^\dag(c),d \rangle$}; 
\node (K) at (0,0) {$\langle FG^{\dag\dag}(c),d \rangle$}; 
\node (L) at (2.5,0) {$\langle G^{\dag\dag}(c),F^\dag(d) \rangle$}; 
\node (M) at (5.75,0) {$\overline{\langle F^\dag(d),G^{\dag\dag}(c) \rangle}$}; 
\node (N) at (8.5,0) {$\overline{\langle d,F^{\dag\dag}G^{\dag\dag}(c) \rangle}$}; 
\node (O) at (11.25,0) {$\langle F^{\dag\dag}G^{\dag\dag}(c),d \rangle$}; 
\draw[->] (A) -- (B);
\draw[->] (B) -- (C);
\draw[->] (C) -- (D);
\draw[->] (D) -- (E);
\draw[->] (F) -- (G);
\draw[->] (G) -- (H);
\draw[->] (H) -- (I);
\draw[->] (I) -- (J);
\draw[->] (K) -- (L);
\draw[->] (L) -- (M);
\draw[->] (M) -- (N);
\draw[->] (N) -- (O);
\draw (A) -- (F);
\draw[<-] (B) --node[left]{$\scriptstyle {\langle\,\,, (\nu_{G,F})_d\rangle}$} (G);
\draw[<-] (C) --node[left]{$\scriptstyle \overline{\langle (\nu_{G,F})_d,\,\,\rangle}$} (H);
\draw[<-] (D) --node[left]{$\scriptstyle {\overline{\langle \,\,,(\nu_{G,F}^\dag)_c\rangle}}$} (I);
\draw[<-] (E) --node[left]{$\scriptstyle \langle (\nu_{G,F}^\dag)_c,\,\,\rangle$} (J);
\draw[->] (F) --node[left]{$\scriptstyle \varphi_G$} (L);
\draw[-] (H) -- (M);
\draw[<-] (I) --node[left]{$\scriptstyle \overline{\langle \,\,,(\nu_{F^\dag,G^\dag})_c\rangle}$} (N);
\draw[<-] (J) --node[left]{$\scriptstyle \langle (\nu_{F^\dag,G^\dag})_c,\,\,\rangle$} (O);
\draw[->] (A) --node[left]{$\scriptstyle \varphi_G$} (K);
\draw[->] (A) to[bend left=18]node[above]{$\scriptstyle \varphi_{FG}$} (E);
\draw[->] (K) to[bend left=-18]node[below]{$\scriptstyle \varphi_F$} (O);
\end{tikzpicture}\vspace{-5mm}
\]
\end{proof}

\section{Positive cones}
\label{sec: cones}

Given a von Neumann algebra $A$, both $A$ and its $L^2$ space are equipped with canonical positive cones
$P_A\subset A$, and $P_{L^2A}\subset L^2A$ \cite{MR0407615}.

\begin{defn}
If $F:C\to D$ is a functor between $\rmW^*$-categories, 
the \emph{vertical cone}\label{pageref vertical cone}
\[
P^{\mathsf v}_F\subset \End(F)
\]
is the positive cone $P_{\End(F)}$ of the 
von Neumann algebra $\End(F)$. It is a subset of the set of natural transformations $\alpha:F\Rightarrow F$ that are `vertically self-adjoint' in the sense that $\alpha^*=\alpha$.
\end{defn}


\begin{defn}
\label{def:  horizontal cone}
If $F:C\to D_1$ and $G:C\to D_2$ are functors between Cauchy complete $\rmW^*$-categories,
the \emph{horizontal cone} 
\label{pageref horizontal cone}
\begin{equation}\label{eq: horizontal cone}
P^{\mathsf h}_{F,G}\subset \Hom(F^\dag\circ F,G^\dag\circ G)
\end{equation}
is the set of natural transformations $\alpha:F^\dag F\Rightarrow G^\dag G$ with the property that for every $c\in C$, the map
\begin{equation}\label{eq: long map 1}
\hspace{-1.7mm}
L^2\big(\End(F(c))\big)\cong
\langle F(c),F(c)\rangle \to
\langle c,F^\dag F(c)\rangle
\xrightarrow{\alpha}
\langle c,G^\dag G(c)\rangle
\to
\langle G(c),G(c)\rangle
\cong
L^2\big(\End(G(c))\big)
\end{equation}
sends  
$P_{L^2\End(F(c))}$
to 
$P_{L^2\End(G(c))}$.
It is a subset of the set of natural transformations $\alpha:F^\dag F\Rightarrow G^\dag G$ that are `horizontally self-adjoint' in the sense that $\alpha^\dagger=\alpha$, where the latter is an abbreviation for the condition
\[
\alpha^\dagger=\nu_{G,G^\dag}\,(\id_{G^\dag}\circ\varphi_G)\,\alpha\,
(\id_{F^\dag}\circ\varphi^{-1}_F)\,\nu_{F,F^\dag}^{-1}.
\]
\end{defn}

\begin{rem}
If $F:C\to D$ is a functor between Cauchy complete $\rmW^*$-categories, then 
the vector space
$\End(F^\dag\circ F)$ admits two \emph{distinct} positive cones: $P^{\mathsf v}_{F^\dag \circ F}$, and $P^{\mathsf h}_{F,F}$.
\end{rem}


Let us now specialize to the case $C=A$-$\Mod$, $D_1=B_1$-$\Mod$, and $D_2=B_2$-$\Mod$. 
Then, by Lemmas~\ref{prop: func to bimod} and~\ref{lem: conjugate bimodule}, functors $F:C\to D_1$, $G:C\to D_2$ as in Definition~\ref{def:  horizontal cone} correspond to bimodules ${}_{B_1}X_A$ and ${}_{B_2}Y_A$, and
natural transformations $\alpha:F^\dag \circ F\Rightarrow G^\dag \circ G$ correspond to bimodule maps
\begin{equation}\label{eq: horizontal cone -- bimodule maps}
\theta: 
{}_A \overline{X}\boxtimes_{B_1} X_A
\longrightarrow 
{}_A \overline{Y}\boxtimes_{B_2} Y_A.
\end{equation}

\begin{defn}[{\cite[Def.~5.5]{MR4581741}}]
Given von Neumann algebras $A,B_1,B_2$
and bimodules ${}_{B_1}X_A$ and ${}_{B_2}Y_A$,
the \emph{horizontal cone}
\begin{equation}\label{eq: horizontal cones number 2}
P^{\mathsf h}_{X,Y}\subset \Hom(
{}_A \overline{X}\boxtimes_{B_1} X_A
, 
{}_A \overline{Y}\boxtimes_{B_2} Y_A
)
\end{equation}
is the set of bimodule maps $\theta$ as in \eqref{eq: horizontal cone -- bimodule maps}
with the property that $\forall n\in\bbN$
the map $\id_{\overline{\bbC^n}}\otimes \theta \otimes \id_{\bbC^n}$
sends $P_{X,n}$ to $P_{Y,n}$,
where
\begin{align*}
P_{X,n}&:=P_{L^2(\End({}_{B_1}X^{\oplus n}))}
\\&\subset 
L^2(\End({}_{B_1}X^{\oplus n}))
\underset{\eqref{eq: <X,X> is L^2}}\cong
\langle {}_{B_1}X^{\oplus n},{}_{B_1}X^{\oplus n}\rangle
\underset{\eqref{eq: <X,Y> as boxtimes}}\cong
\overline{X}^{\oplus n}\boxtimes_{B_1} X^{\oplus n}
=
\overline{\bbC^n} \otimes \overline{X}\boxtimes_{B_1} X \otimes \bbC^n,
\end{align*}
and similarly for $P_{Y,n}$.
\end{defn}

\begin{prop}
A natural transformation $\alpha:F^\dag\circ F\Rightarrow G^\dag\circ G$ belongs to 
$P^{\mathsf h}_{F,G}$
(defined in \eqref{eq: horizontal cone})
if and only if
the corresponding bimodule map
$\theta: {}_A \overline{X}\boxtimes_{B_1} X_A
\to
{}_A \overline{Y}\boxtimes_{B_2} Y_A$
belongs to 
$P^{\mathsf h}_{X,Y}$
(defined in \eqref{eq: horizontal cones number 2}).
\end{prop}

\begin{proof}
If $\alpha\in P^{\mathsf h}_{F,G}$, then setting $c:=(L^2A)^{\oplus n}$, the condition that \eqref{eq: long map 1} maps $P_{L^2\End(F(c))}$
to 
$P_{L^2\End(G(c))}$ is exactly the same as the condition that $\id_{\overline{\bbC^n}}\otimes \theta \otimes \id_{\bbC^n}$ maps $P_{X,n}$ to $P_{Y,n}$.

Conversely, 
if $\theta\in P^{\mathsf h}_{X,Y}$,
to show that the corresponding natural transformation $\alpha:F^\dag\circ F\Rightarrow G^\dag\circ G$ lies in $P^{\mathsf h}_{F,G}$,
we must argue that for every $c={}_AH\in A$-$\Mod$,
the map \eqref{eq: long map 1}
sends  
$P_{L^2\End(F(c))} = P_{L^2\End({}_{B_1}X\boxtimes_A H)}$
to 
$P_{L^2\End(G(c))} = P_{L^2\End({}_{B_2}Y\boxtimes_A H)}$.
By \cite[Lem.~5.8]{MR4581741}, 
$\id_{\overline{H}}\otimes \theta \otimes \id_H \in P^{\mathsf h}_{X\boxtimes_A H,Y\boxtimes_A H}$.
In particular (setting $n=1$ in the definition of $P^{\mathsf h}_{X\boxtimes_A H,Y\boxtimes_A H}$), $\id_{\overline{H}}\otimes \theta \otimes \id_H$ maps $P_{L^2(\End({}_{B_1}X\boxtimes_AH))}$
to
$P_{L^2(\End({}_{B_2}Y\boxtimes_AH))}$, which is exactly what we needed to show.
\end{proof}

\section{Small \texorpdfstring{$\rmW^*$}{W*}-categories}
\label{sec: Small W* categories}

The theory of $\rmW^*$-categories admits a largely parallel version involving only small categories.
Throughout this section, $\kappa$ denotes an infinite cardinal.
We call a Hilbert space $\kappa$-\emph{separable}\label{pageref k-separable} if it admits a basis of cardinality $<\kappa$, and
we call a von Neumann algebra $\kappa$-\emph{separable} if it admits a faithful representation on a $\kappa$-separable Hilbert space.
A von Neumann algebra $A$ is $\kappa$-separable if and only if $L^2A$ is $\kappa$-separable (\cite[Lem~1.8]{MR2325696}).
Setting $\kappa=\aleph_1$ recovers the usual notions of separability for Hilbert spaces and von Neumann algebras.
(And for $\kappa=\aleph_0$, the condition of $\kappa$-separability just amounts to finite dimensionality.)

\begin{defn}\label{def: k-small}
An object in a $\rmW^*$-category is called \emph{$\kappa$-small} if its endomorphism algebra is $\kappa$-separable.
\end{defn}

\begin{lem}\label{lem: small <=> separable}
If $A$ is a $\kappa$-separable von Neumann algebra, then an $A$-module is $\kappa$-small if and only if its underlying Hilbert space is $\kappa$-separable.
\end{lem}

\begin{proof}
Let $H$ be an $A$-module.
If $H$ is $\kappa$-separable, then it is clearly $\kappa$-small (regardless of whether $A$ is $\kappa$-separable).

If $H$ is not $\kappa$-separable then, by transfinite induction, we may find a 
collection of non-zero projections $\{p_\alpha\in A\}_{\alpha<\beta}$, indexed by the set of ordinals smaller than some given ordinal $\beta$, such that there is an isomorphism of $A$-modules
\[
H\,\cong\, \bigoplus_{\alpha<\beta} L^2Ap_\alpha,
\]
and such that 
$p_{\alpha_2} \le p_{\alpha_1}$ whenever $\alpha_2\ge \alpha_1$.
Since $H$ is not $\kappa$-separable while $L^2A$ is, the cardinality of $\beta$ must be $\ge \kappa$ (here, we have used the fact that the cardinality of an infinite set is equal to that of its square).
For each $\alpha\le\beta$, let $q_\alpha:=(\inf_{\alpha'<\alpha}p_{\alpha'})-p_\alpha$,
where we use the convention that $p_\beta=0$.
Let $S:=\{\alpha\le\beta:q_\alpha\neq 0\}$.
Then $\beta=\sup_{\alpha\in S}\alpha$, and
\[
H\,\cong\, \bigoplus_{\alpha \in S} L^2Aq_\alpha\otimes K_\alpha,
\]
where $K_\alpha$ is a Hilbert space of dimension $|\alpha|$.
It follows that $A'$ contains a subalgebra isomorphic to $\bigoplus_{\alpha \in S}B(K_\alpha)$.
Since $\sup_{\alpha\in S}|\alpha|=|\beta|\ge \kappa$, the von Neumann algebra $\bigoplus_{\alpha \in S}B(K_\alpha)$ is not $\kappa$-separable. Hence neither is $A'$,
and $H$ is not $\kappa$-small as an $A$-module.
\end{proof}

\begin{cor}\label{cor: small <=> separable '}
If $A=\prod_{i\in I} A_i$ is a product of $\kappa$-separable von Neumann algebras, then an $A$-module is $\kappa$-small if and only if its underlying Hilbert space is $\kappa$-separable.
\end{cor}

\begin{proof}
Let $H_i$ be $A_i$-modules, and let $H=\bigoplus_{i\in I} H_i$ be the corresponding $A$-module.
If $H$ is $\kappa$-separable, then it is clearly $\kappa$-small.
If $H$ is not $\kappa$-separable then, by the infinite pigeonhole principle, either one of the $A_i$ is not $\kappa$-separable, or $|\{i\in I:H_i\not =0\}|\ge\kappa$.
In either case, $\End_A(H)$ fails to be $\kappa$-separable.
\end{proof}

Recall that all $\rmW^*$-categories are assumed to admit a set of generators.

\begin{defn}\label{def: (locally) kappa-small}
A $\rmW^*$-category $C$ is called \emph{locally $\kappa$-small} if all its objects are $\kappa$-small.
It is called \emph{$\kappa$-small} if all its objects are $\kappa$-small, and $C$ admits a set $\{c_i\}_{i\in I}$ of generators with $|I|<\kappa$.\footnote{For $\kappa=\aleph_1$, these notions are called `separable' and `locally separable' in \cite[Def.~1.9]{MR2325696}.}
\end{defn}

The above terminology is justified by the fact that all
locally $\kappa$-small $\rmW^*$-categories are essentially small:

\begin{lem}
If $C$ is a locally $\kappa$-small $\rmW^*$-category, then its isomorphism classes of objects form a set (as opposed to a proper class).
\end{lem}

\begin{proof}
Let $\{c_i\}_{i\in I}$ be a set of generators,
let $A_i:=\End(c_i)^{\op}$, and let
$A:=\prod A_i$.
By Corollary~\ref{cor: C full subcategory of C_0 plus} and Proposition~\ref{prop:freydembedding'},
$C$ is a full subcateogory of $A$-$\Mod$. So it's enough to show that 
$A$-$\Mod$ only has a set worth of isomorphism classes of $\kappa$-small objects.
By Corollary~\ref{cor: small <=> separable '}, these are the same as $\kappa$-separable modules
and, indeed, actions of $A$ on $\kappa$-separable Hilbert spaces form a set.
\end{proof}

Functor categories between $\kappa$-small $\rmW^*$-categories are typically not $\kappa$-small,
but they are always locally $\kappa$-small 
(see Lemma~\ref{lem: functors categories locally small} below).
If $C$ is a $\rmW^*$-category, then we write $C_{<\kappa}$\label{pageref: C_<k} for its full subcategory of $\kappa$-small objects.
By definition, $C_{<\kappa}\subset C$ is the largest locally $\kappa$-small subcategory of $C$.
The Hilb-valued inner product of $\kappa$-small objects is always $\kappa$-separable:
\[
\langle\,\,,\,\rangle_{\Hilb}:\overline{C_{<\kappa}}\times C_{<\kappa} \to \Hilb_{<\kappa}.
\]

\begin{defn}\label{def: (locally) k-generated}
A $\rmW^*$-category $C$ is called \emph{locally $\kappa$-generated} if it admits a set $\{c_i\}_{i\in I}$ of \emph{$\kappa$-small} generators.
It is called \emph{$\kappa$-generated} if we may furthermore chose $|I|<\kappa$.
\end{defn}

\begin{exs*}$\,$ \vspace{-1mm}
\begin{itemize}
\item
If $A$ is a $\kappa$-separable von Neumann algebra, then the category of $A$-modules is $\kappa$-generated. 
Similarly, $A$-$\Mod_{<\kappa}$ is $\kappa$-generated.
\item
If $A$ is a $\kappa$-separable $\rmC^*$-algebra (one which admits a subset of cardinality $<\kappa$ that linearly spans a dense subspace), then $A$-$\Mod$ and $A$-$\Mod_{<\kappa}$ are locally $\kappa$-generated.
Indeed, any cyclic $A$-module is $\kappa$-separable, hence $\kappa$-small, and those generate $A$-$\Mod$.
The categories
$A$-$\Mod$ and $A$-$\Mod_{<\kappa}$ are typically not $\kappa$-generated.
\item
If $G$ is a group of cardinality $<\kappa$ then $\Rep(G)$ and $\Rep(G)_{<\kappa}$ are locally $\kappa$-generated.
\item
If $A$ and $B$ is a $\kappa$-separable von Neumann algebras, then the categories $\mathrm{Bim}(A,B)$ and $\mathrm{Bim}(A,B)_{<\kappa}$ are locally $\kappa$-generated.
Once again, cyclic $A$-$B$-bimodules are $\kappa$-separable hence $\kappa$-small, and those generate $\mathrm{Bim}(A,B)$.
\end{itemize}
\end{exs*}

\noindent
We say that a $\rmW^*$-category $C$ \emph{admits all $\kappa$-small direct sums}\label{pageref: to admit all k-small direct sums} if every collection of objects $\{c_i\}_{i\in I}$ with $|I|<\kappa$ admits an orthogonal direct sum in $C$.
We say that $C$ is $\kappa$-Cauchy complete\label{pageref: k-Cauchy complete} if it is idempotent complete, and admits all $\kappa$-small direct sums.
Clearly, if $D$ is $\kappa$-Cauchy complete, then so is $\mathrm{Func}(C,D)$.

Following Terminology~\ref{rem: terminology complete}, we will use `$\kappa$-complete' as a synonym of `$\kappa$-Cauchy complete'.

\begin{rem}
For $\kappa=\aleph_0$, a locally $\kappa$-small $\kappa$-complete $\rmW^*$-category is the same thing as a semisimple $\rmC^*$-category \cite[\S2.1]{MR4750417}, and a
$\kappa$-small $\kappa$-complete $\rmW^*$-category is the same thing as a semisimple $\rmC^*$-category with finitely many types of simple objects.
\end{rem}

\begin{defn}\label{def: k-Cauchy completion}
Let $C$ be a locally $\kappa$-small $\rmW^*$-category.
Then the $\kappa$-completion of $C$ is the 
full subcategory
$C^{\hat \oplus_\kappa}\subset C^{\hat \oplus}=
(\hat C)^{\bar\oplus}$
on those objects which can be written as $\kappa$-small orthogonal direct sums of objects of $\hat C$.
\end{defn}

We have the following straightforward analog of Corollary~\ref{cor: Cauchy completion}:
\begin{prop}
\label{prop: Cauchy completion - kappa}
Let $C$ and $D$ be $\rmW^*$-categories, with $D$ $\kappa$-complete.
Then the restriction functor
\[
\mathrm{Func}\big(C^{\hat \oplus_\kappa},D\big) \to \mathrm{Func}\big(C,D\big)
\]
is an equivalence of categories.\hfill $\square$
\end{prop}

The following is an analog of Lemma~\ref{lem: canonical  functor an equivalence iff C_0 generates C}:

\begin{lem}
\label{lem: canonical  functor an equivalence iff C_0 generates C - kappa}
Let $C$ be a locally $\kappa$-small $\kappa$-complete $\rmW^*$-category, and let $C_0\subset C$ be a full subcategory.
Then the canonical  functor $F:(C_0)^{\hat \oplus_\kappa}\to C$ provided by Proposition~\ref{prop: Cauchy completion - kappa} is fully faithful. It is an equivalence if and only if $C_0$ generates $C$.
\end{lem}

\begin{proof}
Identical to that of Lemma~\ref{lem: canonical  functor an equivalence iff C_0 generates C}
(using the fact that if $c\in C$ is an object, then a family of non-zero partial isometries $\{v_i: x_i\to c\}$ with domains $x_i\in C_0$ and orthogonal ranges cannot have cardinality $\ge \kappa$).
\end{proof}

Given a collection $\{C_i\}_{i\in I}$ of $\rmW^*$-categories, let us write $\boxplus_{i\in I}^{<\kappa} C_i\subset \boxplus_{i\in I} C_i$\label{pageref boxplus^<k} for the full subcategory on whose elements $\boxplus_{i\in I} c_i$ satisfying $|\{i\in I:c_i\neq 0\}|<\kappa$.
The following is an analog of Propositions~\ref{prop:freydembedding}
and~\ref{prop:freydembedding'}:

\begin{prop}
\label{prop:freydembedding''}
{\it (i)}
If $C$ is a $\kappa$-small $\kappa$-complete $\rmW^*$-category, then $C$ is equivalent to the category of $\kappa$-separable modules over some $\kappa$-separable von Neumann algebra.
Moreover, if $c\in C$ is a generator and $A:=\overline{\End(c)}$, then
\begin{align}
\notag
F:C&\to A\text{-}\Mod_{<\kappa}
\\
\label{eq: freydembedding' - kappa}
x&\mapsto\;\! \langle c,x\rangle
\end{align}
provides an equivalence.

{\it (ii)}
If $C$ is only locally $\kappa$-small, then it is equivalent to the category of $\kappa$-separable modules over a von Neumann algebra which is a product of $\kappa$-separable algebras.
Moreover, if $c_i\in C$ are orthogonal generators, $A_i=\overline{\End(c_i)}$, and $A=\prod A_i$, then 
\begin{align}
\notag
F:C&\to A\text{-}\Mod_{<\kappa}
\\
\label{eq: freydembedding - locally kappa}
x&\mapsto\;\! 
\textstyle\bigoplus_i \langle c_i,x\rangle
\end{align}
is an equivalence of categories.

\end{prop}
\begin{proof}
The proof of Proposition~\ref{prop:freydembedding'} applies verbatim to prove part {\it (i)}.
We must however provide a different argument for part {\it (ii)} as $C$ might only admit a set of generators, but not a single generator (the direct sum of all the generators might not exist in $C$).

By Proposition \ref{prop: Gram–Schmidt} (the Gram–Schmidt process for $\rmW^*$-categories) adapted to the context of $\kappa$-small categories,
there exists an orthogonal set of generators $\{c_i\}_{i\in I}$ and an orthogonal direct sum decomposition $C=\boxplus^{<\kappa}_iC_i$ such that $C_i$ is the subcategory of $C$ generated by $c_i$.
The functor $C_i\to A_i\text{-}\Mod_{<\kappa}:
x\mapsto \langle c_i,x\rangle$ is an equivalence by part {\it (i)} and the result follows since
$
A\text{-}\Mod_{<\kappa}
=
(\prod A_i)\text{-}\Mod_{<\kappa}
\cong
\boxplus^{<\kappa}_i(A_i\text{-}\Mod_{<\kappa})
$.
\end{proof}

\begin{cor}
Let $C$ be a locally $\kappa$-small $\kappa$-complete $\rmW^*$-category. Then $C\cong \boxplus_{i\in I}^{<\kappa} C_i$ for some set $I$, and some collection $C_i$ of $\kappa$-small $\kappa$-complete $\rmW^*$-categories.
\end{cor}

So far, $\kappa$ could be taken to be an arbitrary infinite cardinal. The next two results require $\kappa$ to be a successor cardinal, and are false for limit cardinals.
The following is a refinement of \cite[Cor~7.16]{MR808930}:

\begin{lem}\label{lem: GLR Cor 7.16}
Let $\kappa=\aleph_{\alpha+1}$ be a successor cardinal, and let $C$ be a $\kappa$-small $\kappa$-complete $\rmW^*$-category.

Then there exists a generator $c\in C$ 
such that every object of $C$ is isomorphic to a subobject of $c$.
Such an object is unique up to isomorphism.
\end{lem}

\begin{proof}
By combining part {\it (i)} of Proposition~\ref{prop:freydembedding''}
with Corollary~\ref{cor: submodules of opplus of faithful A-module}, there exists a generator $x\in C$ such that
every object $y\in C$
admits a direct sum decomposition $y=\bigoplus_{i\in I} y_i$ where each $y_i$ embeds into $x$.
Assuming without loss of generality that none of the $y_i$ is zero,
the indexing set $I$ must have cardinality $<\kappa$, as otherwise this would contradict the fact that $y$ is $\kappa$-small.
So we get an embedding $y \hookrightarrow x^{\oplus I}$, for some set $I$ of cardinality $|I|<\aleph_{\alpha+1}$.
In particular $y \hookrightarrow x^{\oplus \aleph_{\alpha}}$. So $c:=x^{\oplus \aleph_{\alpha}}$ is the desired generator.

The uniqueness claim follows from Lemma~\ref{lem:Cantor-Schroeder-Bernstein}.
\end{proof}

We call an object satisfing the condition in Lemma~\ref{lem: GLR Cor 7.16} a \emph{maximal} object\label{pageref: maximal}  \cite[\S6]{MR808930}.

\begin{cor}
Let $\kappa=\aleph_{\alpha+1}$ be a successor cardinal, and let $C$ be a $\kappa$-small $\kappa$-complete $\rmW^*$-category.

Then there exists a $\kappa$-separable von Neumann algebra $A$ such that
$C$ is equivalent to $(\bfB A)\hspace{-.5mm}\hat{\phantom t}$
(idempotent completion, but no direct sum completion).
\end{cor}

\begin{proof}
Let $A:=\End(c)$ for $c\in C$ maximal. Then argue as in \cite[Prop~6.4]{MR808930}.
\end{proof}

\begin{lem}
\label{lem: functors categories locally small}
If $C$ is $\kappa$-small,
and $D$ is locally $\kappa$-small,
then $\mathrm{Func}(C,D)$ is locally $\kappa$-small.
\end{lem}

\begin{proof}
$D$ is a full subcategory of $D^{\hat \oplus_\kappa}$, so $\mathrm{Func}(C,D)$
is a full subcategory of $\mathrm{Func}(C,D^{\hat \oplus_\kappa})$.
The latter is equivalent to
$\mathrm{Func}(C^{\hat \oplus_\kappa},D^{\hat \oplus_\kappa})$ by Proposition~\ref{prop: Cauchy completion - kappa}.
So we may assume without loss of generality that both $C$ and $D$ are $\kappa$-complete.

By Proposition~\ref{prop:freydembedding''}, 
we may then write $C=A\text{-}\Mod_{<\kappa}$
and $D=B\text{-}\Mod_{<\kappa}$,
for $A$ a $\kappa$-separable von Neumann algebra,
and $B=\prod B_i$ a product of $\kappa$-separable von Neumann algebras.
By the same argument as above (using Corollary~\ref{cor: Cauchy completion} instead of Proposition~\ref{prop: Cauchy completion - kappa}), along with Proposition~\ref{prop: func to bimod}, $\mathrm{Func}(C,D)$
is equivalent to a full subcategory of $\mathrm{Func}(C^{\hat \oplus},D^{\hat \oplus})\cong\Bim(B,A)$.
Specifically, it is the full subcategory on those functors $F:C^{\hat \oplus}\to D^{\hat \oplus}$ that map 
$(C^{\hat \oplus})_{<\kappa}=C$ to
$(D^{\hat \oplus})_{<\kappa}=D$.

For any such functor $F$, 
since $L^2A\in C$, we must have $F(L^2A)\in D$.
By Corollary~\ref{cor: small <=> separable '},  the bimodule ${}_BH_A$ that corresponds to $F$
must therefore have its underlying $B$-module ${}_BH={}_BH\boxtimes_a L^2A=F(L^2A)$ be
$\kappa$-separable.
We finish the proof by noting that any $\kappa$-separable bimodule is $\kappa$-small.
\end{proof}

The next result is immediate from the definitions, and from Lemma~\ref{lem: canonical  functor an equivalence iff C_0 generates C}:

\begin{lem}
Let $C$ be a Cauchy-complete $\rmW^*$-category.
If $C$ is locally $\kappa$-generated, then $C=(C_{<\kappa})^{\hat \oplus}$.
If $C$ is furthermore $\kappa$-generated, then $C_{<\kappa}$ is $\kappa$-small. 
\hfill $\square$
\end{lem}

If $C$ and $D$ are $\kappa$-generated Cauchy-complete $\rmW^*$-categories, then $\mathrm{Func}(C,D)$
is typically not $\kappa$-generated.
Instead, we have:

\begin{prop}
\label{prop: functors between small categories}
Let $C$ and $D$ be Cauchy-complete $\rmW^*$-categories
such that $C$ is $\kappa$-generated and $D$ is locally $\kappa$-generated.
Then $\mathrm{Func}(C,D)$ is locally $\kappa$-generated, and
\begin{equation}
\label{eq: EQUIV 1}
\mathrm{Func}\big(C_{<\kappa},D_{<\kappa}\big) \cong \mathrm{Func}\big(C,D\big)_{<\kappa}.
\end{equation}
\end{prop}

\begin{proof}
By Proposition~\ref{prop:freydembedding'}, 
we may assume without loss of generality that $C=A\text{-}\Mod$
and $D=B\text{-}\Mod$,
for $A$ a $\kappa$-separable von Neumann algebra,
and $B=\prod B_i$ a product of $\kappa$-separable von Neumann algebras.
By Proposition~\ref{prop: func to bimod}, we then have
$\mathrm{Func}(C,D)= \Bim(B,A)=\boxplus_i \Bim(B_i,A)$.
Since every $B_i$-$A$-bimodule is a sum of cyclic bimodules, and every cyclic $B_i$-$A$-bimodule is $\kappa$-separable, $\Bim(B_i,A)$ is locally $\kappa$-generated (generated by the set of $\kappa$-separable bimodules).
$\Bim(B,A)$ is therefore also locally $\kappa$-generated.

Let $\{X_i\}_{i\in I}$ be a set of $\kappa$-small generators for $\Bim(B,A)$.
By Proposition~\ref{prop: Gram–Schmidt} (Gram– Schmidt for $\rmW^*$-categories), we may arrange 
for them to form an orthogonal set of generators.
Then $X:=\bigoplus_{i\in I} X_i$ is a generator of $\Bim(B,A)$ whose endomorphism algebra is a product of $\kappa$-separable von Neumann algebras.
By Proposition~\ref{prop:freydembedding'} and Corollary~\ref{cor: small <=> separable '}, it follows that a $B$-$A$-bimodule is $\kappa$-small if and only if it is $\kappa$-separable.

If ${}_BH_A$ is $\kappa$-separable bimodule, then
${}_BH\boxtimes_A-:A\text{-}\Mod\to B\text{-}\Mod$ sends $\kappa$-separable $A$-modules to $\kappa$-separable $B$-modules. This
establishes one direction of the desired equivalence:
\[
\mathrm{Func}\big(A\text{-}\Mod,B\text{-}\Mod\big)_{<\kappa}=\Bim(B,A)_{<\kappa} \to
\mathrm{Func}\big(A\text{-}\Mod_{<\kappa},B\text{-}\Mod_{<\kappa}\big).
\]
The inverse functor 
$\mathrm{Func}\big(A\text{-}\Mod_{<\kappa},B\text{-}\Mod_{<\kappa}\big) \to \Bim(B,A)_{<\kappa}$
is given by evaluating on $L^2A\in A\text{-}\Mod_{<\kappa}$, as in the statement of Proposition~\ref{prop: func to bimod}.
\end{proof}

\begin{cor}
If $C$ and $D$ are $\kappa$-generated Cauchy-complete $\rmW^*$-categories, then\linebreak 
$\dagger|_{\mathrm{Func}(C,D)_{<\kappa}}:\mathrm{Func}(C,D)_{<\kappa}\to \mathrm{Func}(D,C)$
lands in $\mathrm{Func}(D,C)_{<\kappa}$, and induces an antilinear equivalence
\[
\dagger\,:\,\mathrm{Func}\big(C_{<\kappa},D_{<\kappa}\big)\to \mathrm{Func}\big(D_{<\kappa},C_{<\kappa}\big).
\]
This operation satisfies the same defining property
\eqref{eq: defining property of adjoint} as \eqref{eq: dag}.
$\,$ \hfill $\square$
\end{cor}

\begin{prop}
If $C$ and $D$ are $\kappa$-complete $\rmW^*$-categories where $C$ is $\kappa$-small, and $D$ is locally $\kappa$-small. Then
\begin{equation}
\label{eq: EQUIV 2}
\mathrm{Func}\big(C^{\hat\oplus},D^{\hat\oplus}\big) \cong \mathrm{Func}\big(C,D\big)^{\hat\oplus}.
\end{equation}
\end{prop}

\begin{proof}
\begin{align*}
\mathrm{Func}\big(C,D\big)^{\hat\oplus}
&\cong
\mathrm{Func}\big((C^{\hat\oplus})_{<\kappa},(D^{\hat\oplus})_{<\kappa}\big)^{\hat\oplus}\\
&\cong
\big(
\mathrm{Func}\big((C^{\hat\oplus}),(D^{\hat\oplus})\big)_{<\kappa}
\big)^{\hat\oplus}
\cong
\mathrm{Func}\big(C^{\hat\oplus},D^{\hat\oplus}\big),
\end{align*}
where the middle equivalence holds by Proposition~\ref{prop: functors between small categories},
and the last one holds because $\mathrm{Func}((C^{\hat\oplus}),(D^{\hat\oplus}))$
is locally $\kappa$-generated (again by Proposition~\ref{prop: functors between small categories}).
\end{proof}

Taken together, \eqref{eq: EQUIV 1} and \eqref{eq: EQUIV 2} establish an equivalence of $2$-categories between the $2$-category of 
$\kappa$-generated complete $\rmW^*$-categories\footnote{Note the use of Terminology~\ref{rem: terminology complete}.}
and $\kappa$-small functors between them,
and the $2$-category of
$\kappa$-small $\kappa$-complete $\rmW^*$-categories and all functors:
\begin{equation}
\label{eq: two 2-categoies}
\begin{tikzpicture}[baseline = -2]
\node[left] at (-1.2,0)
{$\left\{
\parbox{4.3cm}{\centerline{$\kappa$-generated complete}\centerline{$\rmW^*$-categories}}
\right\}$};
\node[right] at (1.2,0)
{$\left\{
\parbox{4cm}{\centerline{$\kappa$-small $\kappa$-complete} \centerline{$\rmW^*$-categories}}
\right\}$};
\draw[->] (-1.1,.2) --node[above]{$(\;\;\,)_{<\kappa}$} (1.1,.2);
\draw[<-] (-1.1,-.2) --node[below]{$(\;\;\,)^{\hat\oplus}$} (1.1,-.2);
\end{tikzpicture}
\end{equation}
Moreover, any functor between $\kappa$-generated complete $\rmW^*$-categories is a direct sum of $\kappa$-small functors.
Both sides of \eqref{eq: two 2-categoies} carry an involution $\dagger$ at the level of $1$-morphisms, and an involution $*$ at the level of $2$-morphisms, and the $2$-functors in \eqref{eq: two 2-categoies} are compatible with both $\dagger$ and $*$.


\section{\texorpdfstring{$\rmW^*$}{W*}-tensor categories}

A \emph{$\rmW^*$-tensor category}\label{pageref: W^*-tensor category} $(T,\otimes,1,\alpha,l,r)$ is a $\rmW^*$-category with a monoidal structure which is compatible with the $\rmW^*$-category structure in the sense that
\[
\otimes: T \times T \to T
\]
is a bilinear functor of $\rmW^*$-categories (Definition~\ref{def: functors linear and bilinear}),
and the associator $
\alpha$ and left and right unitors $l,r$ are unitary.

A \emph{tensor functor} between two $\rmW^*$-tensor categories $S$ and $T$, is a triple $(F,\mu,i)$, where $F:S\to T$ is a functor of $\rmW^*$-categories, and
\begin{equation}
\label{eq: mu and i coherences}
\mu_{x,y}:F(x) \otimes_{T} F(y) \to F(x \otimes_{S} y)\qquad\text{and}\qquad
i:1_{T}\to F(1_{S})
\end{equation}
are unitary natural transformations satisfying     $\mu_{x,y\otimes z}\circ(\id_{F(x)}\otimes\mu_{y,z}) = \mu_{x\otimes y, z}\circ (\mu_{x,y}\otimes \id_{F(z)})$ and  $\mu_{1,x}\circ (i\otimes \id_{F(x)}) = \id_{F(x)} = \mu_{x,1}\circ(\id_{F(x)}\otimes i)$ (suppressing associators). 
Finally, a \emph{monoidal natural transformation} between two tensor functors $F$ and $G$ is natural transformation $\alpha:F\Rightarrow G$ that intertwines the coherences \eqref{eq: mu and i coherences} of the two tensor functors.

If a $\rmW^*$-tensor category $T$ comes equipped with an action of $\bbZ/2$ in which the non-trivial element of $\bbZ/2$ acts antilinearly and anti-monoidally, then $T$ is called a bi-involutive $\rmW^*$-tensor category:

\begin{defn}
\label{def: bi-involutive tensor category} 
A \emph{bi-involutive $\rmW^*$-tensor category} is a $\rmW^*$-tensor category $T$ equipped with a covariant anti-linear, anti-tensor functor
\begin{equation*}
\overline{\,\cdot\,}:T\to T
\end{equation*}
called the conjugate.
The functor $\overline{\,\cdot\,}$ is involutive, meaning that for every $x\in T$, we are given unitary natural unitary isomorphisms $\varphi_x:x\to \overline{\overline{x}}$ satisfying $\varphi_{\overline{x}}=\overline{\varphi_{x}}$. 
The structure data of this anti-tensor functor are denoted
\begin{equation*}
\nu_{x,y}:\overline{x} \otimes \overline{y} \stackrel\simeq\longrightarrow \overline{y \otimes x}\qquad\text{and}\qquad j:1\to\overline 1
\end{equation*}
and make the following diagrams commute:

 \begin{equation*}
        \begin{tikzcd}
            \overline x\otimes\left(\overline y\otimes\overline  z\right)\arrow[rr,"\mathrm{id}_{\overline x}\otimes\nu_{y,z}"]\arrow[d,"\alpha^{-1}_{\overline x,\overline y,\overline z}"']&&
            \overline x\otimes\overline{z\otimes y}\arrow[rr,"\nu_{x,z\otimes y}"]&&
            \overline{\left(z\otimes y\right)\otimes x}
            \\
            \left(\overline x\otimes\overline y\right)\otimes\overline z\arrow[rr,"\nu_{x,y}\otimes\mathrm{id}_{\overline z}"]&&
            \overline{y\otimes x}\otimes\overline z\arrow[rr,"\nu_{y\otimes x,z}"]&&\overline{z\otimes\left( y\otimes x\right)}\arrow[u,<-,"\overline{\alpha_{z,y,x}}"']
        \end{tikzcd}
    \end{equation*}
    \begin{equation*}
        \begin{tikzcd}
            1\otimes\overline x\arrow[rr,"j\otimes\mathrm{id}_{\overline x}"]\arrow[d,"l_{\overline x}"']&&\overline1\otimes\overline x\arrow[d,"\nu_{1,x}"]&\overline x\otimes1\arrow[rr,"\mathrm{id}_{\overline x}\otimes j"]\arrow[d,"r_{\overline x}"']&&\overline x\otimes\overline1\arrow[d,"\nu_{x,1}"]\\
            \overline x&&\overline{x\otimes1}\arrow[ll,"\overline{r_x}"]&\overline x&&\overline{1\otimes x}\arrow[ll,"\overline{l_x}"]
        \end{tikzcd}
    \end{equation*}
Finally, we require the compatibility conditions $\varphi_1=\overline j\circ j$ and $\varphi_{x \otimes y}=\overline{\nu_{y,x}}\circ\nu_{\overline x,\overline y}\circ(\varphi_x\otimes\varphi_y)$.
\end{defn}

Note that, by Lemma~\ref{lem: antiunitaries at level of <,>}, the antilinear involution $\overline{\,\cdot\,}$ induces canonical unitary isomorphisms $\langle\overline x,\overline y\rangle\to\overline{\langle x,y\rangle}$ for all $x,y\in T$, such that the composite $\langle x,y\rangle\simeq\big\langle \overline{\overline x},\overline{\overline y}\big\rangle\to\overline{\langle\overline x,\overline y\rangle}\to\overline{\overline{\langle x,y\rangle}}\simeq\langle x,y\rangle$ is the identity.
    
\begin{defn}\label{def: bi-involutive tensor functor}
A tensor functor $F\colon S\to T$ between bi-involutive $\rmW^*$-tensor categories is called a \emph{bi-involutive tensor functor} if it comes equipped with a unitary natural transformation
\begin{equation*}
\label{eq: coherence gamma}
\gamma_x:F(\overline x)\to \overline{F(x)}
\end{equation*}
making the following diagrams commute:\footnote{The coherence $\gamma_{1}=\overline i\circ j^{T}\circ i^{-1}\circ F(j^{S})^{-1}$ for the unit follows from the coherence for $\gamma_{x\otimes y}$.}
\begin{equation*}
    \begin{tikzcd}
{F(\overline x)\otimes F(\overline y)}&& {F(\overline x\otimes \overline y)} &&
  {F(\overline{y\otimes x})}
	\\[1mm]
{\overline{F(x)}\otimes \overline{F(y)}} && {\overline{F(y)\otimes F(x)}} && {\overline{F(y\otimes x)}} 
	\arrow["{\gamma_{y\otimes x}}", from=1-5, to=2-5]
	\arrow["{F(\nu^{S}_{x,y})}", from=1-3, to=1-5]
	\arrow["{\mu_{\overline x,\overline y}}", from=1-1, to=1-3]
	\arrow["{\gamma_x \otimes \gamma_y}"', from=1-1, to=2-1]
	\arrow["{\overline{\mu_{y\otimes x}}}", from=2-3, to=2-5]
	\arrow["{\nu^{T}_{F(x)\otimes F(y)}}", from=2-1, to=2-3] 
\end{tikzcd}
\end{equation*} 

\begin{equation*}
   \begin{tikzcd}
	{F(x)} && {F(\overline{\overline x})}  \\
	{\overline{\overline{F(x)}}} &&{\overline{F(\overline x)}}
	\arrow["{F(\varphi^{S}_x)}", from=1-1, to=1-3]
	\arrow["{\gamma_{\overline x}}", from=1-3, to=2-3]
	\arrow["{\varphi^{T}_{F(x)}}"', from=1-1, to=2-1]
	\arrow["{\overline{\gamma_x}}", from=2-3, to=2-1]
\end{tikzcd}
\end{equation*}

Finally, a bi-involutive natural transformation is a monoidal natural transformation that intertwines the coherences \eqref{eq: coherence gamma} of the two tensor functors.
\end{defn}
    
If $C$ is a Cauchy complete $\rmW^*$-category, then the category $\End(C)$ of endo\-functors of $C$ is a bi-involutive $\rmW^*$-tensor category, with involution provided by the operation $F\to F^\dagger$.
Upon identifying $C$ with $A\text{-}\Mod$ for some von Neumann algebra (using Proposition~\ref{prop:freydembedding}), we have an equivalence of tensor categories between $(\End(C), \circ, \id_C)$ and $(\Bim(A), \boxtimes_A, {}_AL^2A_A)$ (by Proposition~\ref{prop: func to bimod}).
The modular conjugation $J$, and the unitary
\begin{equation}\label{eq: unitary nu}
\nu:\overline H\boxtimes_A \overline K\to  \overline{K\boxtimes_A H}
\end{equation}
given by $\nu(\varphi\otimes \xi\otimes \psi):= 
(\psi\circ J)\otimes J(\xi)\otimes(\varphi\circ J)
$ endow $\Bim(A)$ with the structure of a bi-involutive $\rmW^*$-tensor category,
and the equivalence between $\End(C)$ and $\Bim(A)$ intertwines the involution $F\mapsto F^\dagger$ on $\End(C)$ and the involution ${}_AH_A\mapsto {}_A\overline H_A$ on $\Bim(A)$ (by Lemma~\ref{lem: conjugate bimodule}).

\begin{defn}
The `many object' versions of Definitions~\ref{def: bi-involutive tensor category} and~\ref{def: bi-involutive tensor functor} are called \emph{bi-involutive $\rmW^*$-bicate\-gories}, and \emph{bi-involutive $2$-functors}. 
A bi-involutive $2$-functor whose underlying $2$-functor is an equivalence is called an equivalence of bi-involutive $\rmW^*$-bicategories.
\end{defn}

The collection of all von Neumann algebras, bimodules, and bounded linear maps between those forms the objects, $1$-morphisms, and $2$-morphisms of a $\rmW^*$-bicategory $\mathrm{vN2}$\label{pageref: vN2} 
(\cite{10.1063/1.1563733,MR2325696}, see also \cite[\S2.1]{MR4419534}).
Composition of $1$-morphisms is given by Connes fusion\footnote{For composition to be a map $\Hom_{\mathrm{vN2}}(B,C)\times \Hom_{\mathrm{vN2}}(A,B) \to \Hom_{\mathrm{vN2}}(A,C)$, we adopt the convention $\Hom_{\mathrm{vN2}}(A,B):=\Bim(B,A)$.}, and the unit $1$-morphism on an object $A$ is the bimodule ${}_AL^2A_A$. The involution ${}_BX_A\mapsto {}_A\overline{X}_B$ along with the coherence \eqref{eq: unitary nu} at the level of $1$-morphisms, and the $*$-operation at the level of $2$-morphisms equip $\mathrm{vN2}$ with the structure of a bi-involutive $\rmW^*$-bicategory.

The collection of all Cauchy complete $\rmW^*$-categories, functors, and natural transformations (Definition~\ref{def: functors linear and bilinear}) also assemble to $\rmW^*$-bicategory $\mathrm{W^*Cat}$\label{pageref: WstarCat}.
The $\dagger$-operation at the level of $1$-morphisms (Definition~\ref{def: adjoint functors}),
and the $*$-operation at the level of $2$-morphisms (Definition~\ref{def: functor categories}) endow it with the structure of a bi-involutive bicategory.

\begin{thm}\label{thm : vN2=W*Cat}
The assignment $A\mapsto A\text{-}\Mod$ induces an
equivalence of bi-involutive bicategories $\mathrm{vN2}\stackrel{\simeq}\to\mathrm{W^*Cat}$.
\end{thm}

\begin{proof}
The $2$-functor $\mathrm{vN2}\to\mathrm{W^*Cat}$ given by $A\mapsto A\text{-}\Mod$ at the level of objects, and by ${}_BH_A\mapsto ({}_BH\boxtimes_A-)$ at the level of $1$-morphisms
is essentially surjective by part {\it (i)} of Proposition~\ref{prop:freydembedding}, and fully faithful by Proposition~\ref{prop: func to bimod}.
By Lemma~\ref{lem: conjugate bimodule}, it intertwines the involution ${}_BH_A\mapsto {}_A\overline{H}_B$ on $\mathrm{vN2}$ with the involution $F\mapsto F^\dagger$ on $\mathrm{W^*Cat}$.
\end{proof}

\begin{rem}
Both $\mathrm{vN2}$ and $\mathrm{W^*Cat}$ are symmetric monoidal, and the equivalence between them 
upgrades to an equivalence of symmetric monoidal bi-involutive $\rmW^*$-bicategories (see Lemma~\ref{lem:TensoratorForVNA->W*Cat}). The symmetric monoidal structures are given by the spatial tensor product for $\mathrm{vN2}$, and the operation~\eqref{eq: completed tens of W*cat} for $\mathrm{W^*Cat}$.

By the results of Section~\ref{sec: cones}, the equivalence $\mathrm{vN2}\cong \mathrm{W^*Cat}$ is furthermore compatible with the vertical and horizontal cones that exist of both sides.
\end{rem}

Theorem~\ref{thm : vN2=W*Cat} admits a variant in the context of small $\rmW^*$-categories.
Let $\mathrm{vN2}_{<\kappa}$ be the bicategory of $\kappa$-separable von Neumann algebras and $\kappa$-separable bimodules between them, and let 
$\mathrm{W^*Cat}_{<\kappa}$ be the $2$-category of
$\kappa$-small $\kappa$-complete $\rmW^*$-categories.

\begin{thm}\label{thm : vN2=W*Cat kappa}
The assignment $A\mapsto A\text{-}\Mod_{<\kappa}$ induces an equivalence of bi-involutive bicategories $\mathrm{vN2}_{<\kappa}\stackrel{\simeq}\to\mathrm{W^*Cat}_{<\kappa}$. \hfill $\square$
\end{thm}

\begin{rem}
In view of \eqref{eq: two 2-categoies}, $\mathrm{vN2}_{<\kappa}$ is
also equivalent to the $2$-category of
$\kappa$-generated Cauchy-complete $\rmW^*$-categories
and $\kappa$-small functors between them, via the assignment $A\mapsto A\text{-}\Mod$.
\end{rem}

In addition to the two invlutions (at the level of 1-morphisms, and 2-morphimsms) that both $\mathrm{vN2}$ and $\mathrm{W^*Cat}$ possess, there is also an involution at the level of objects given by complex conjugation. The equivalence from Theorem \ref{thm : vN2=W*Cat} is compatible with all three involutions.
Let us write $\dagger_0$ for the involution at the level of object, $\dagger_1$ for the involution at the level of $1$-morphisms, and $\dagger_2$ for the involution at the level of $2$-morphisms. The involutions are given by:\medskip

\centerline{\begin{tabular}{c|c|c|}
& $\mathrm{vN2}$ & $\mathrm{W^*Cat}$ \\[1mm]\hline
$\quad\dagger_0\quad\phantom{\Big|}$ & complex conjugation & complex conjugation \\[1.8mm]\hline
$\dagger_1\phantom{\Big|}$ & complex conjugation & adjoint \\[1.8mm]\hline
$\dagger_2\phantom{\Big|}$ & adjoint & adjoint \\[1mm]\hline
\end{tabular}}\bigskip

\noindent
(The three involutions on $\mathrm{vN2}$ are discussed \cite[\S4.3]{MR2079378}.)
Taken together, these three involutions make $\mathrm{vN2}$ and $\mathrm{W^*Cat}$ `fully dagger' in the sense of \cite{2403.01651}.

\bibliographystyle{alpha}
{\footnotesize{
\bibliography{bibliography}
}}

\newpage
\section*{Appendix}

\centerline{\sc Higher (and lower) categorical analogs of Hilbert spaces} 
\centerline{\sc and von Neumann algebras}\vspace{1cm}
\pagestyle{empty}

\def\h{.17}
\centerline{\begin{tikzpicture}
\node[scale=.9]{
\tikz[thick, yscale=1.1]{
\node[scale=1.1, draw, ellipse, inner ysep=5] (a0) at (0,10) {complex number};
\node[scale=1.1, draw, ellipse, inner ysep=5, inner xsep=7] (b0) at (10,10) {positive number};
\node[scale=1.1, draw, ellipse, inner ysep=5, inner xsep=10] (a1) at (0,5) {Hilbert space};
\node[scale=1.1, draw, ellipse, inner ysep=5, inner xsep=-2] (b1) at (10,5) {von Neumann algebra};
\node[scale=1.1, draw, ellipse, inner ysep=5, inner xsep=12] (a1) at (0,0) {W*-category};
\node[scale=1.1, draw, ellipse, inner ysep=5, inner xsep=-2] (b1) at (10,0) {bicommutant category};
\node[scale=1.4] at (0,-5) {??};
\node[scale=1.4] at (10,-5) {??};
\draw[-stealth'] (3,10+\h) --node[above]{$z\mapsto|z|^2$} (7,10+\h);
\draw[right hook-stealth'] (7,10-\h) -- (3,10-\h);
\draw[-stealth'] (-\h*1.1,6) --	node[xshift=-2, pos=.3,left]{$\xi,\eta\in H$}
					node[xshift=-32, above, rotate=90]{$\mapsto$}
					node[pos=.7,left]{$\langle\xi,\eta\rangle\in \bbC$} (-\h*1.1,9);
\draw[-stealth'] (\h*1.1,9) --	node[xshift=12, pos=.38,right]{$\bbC$ is a}
					node[xshift=1, pos=.56,right]{Hilbert sp.} (\h*1.1,6);
\draw[-stealth'] (10+\h*1.1,6) --	node[xshift=6, pos=.32,right]{$a\in A$}
					node[xshift=22, above, rotate=90]{$\mapsto$}
					node[pos=.7,right]{$\|a\|^2$} (10+\h*1.1,9);
%
\draw[-stealth'] (2.3-\h*.6,6+\h*.8) --node[scale=.9, above, sloped, rotate=3, pos=.48]{\,$H\ni\xi\mapsto \|\xi\|^2$} (7.9-\h*.6,9+\h*.8);
\draw[-stealth'] (7.9+\h*.6,9-\h*.8) --	node[scale=.9, below, sloped, rotate=3, pos=.46, yshift=.5]{$\big(\bbC,\sqrt x\in \bbC\big)\,\mapsfrom\, x$}
							node[yshift=-15, scale=.6, below, sloped, rotate=3, pos=.56]{or $0$ if $x=0$} 
							(2.3+\h*.6,6-\h*.8);
\draw[-stealth'] (3,5+\h) --node[scale=.9, above]{\,$H\mapsto B(H)$} (6.7,5+\h);
\draw[-stealth'] (6.7,5-\h) --node[scale=.9, below]{$L^2A \mapsfrom A$\,} (3,5-\h);
\draw[-stealth'] (-\h*1.1,1) --	node[xshift=-12, pos=.25,left]{$a,b\in C$}
					node[xshift=-42, pos=.45, above, rotate=90]{$\mapsto$}
					node[xshift=-8, pos=.85,left]{$\langle a,b\rangle_{\mathrm{Hilb}}:=$}
					node[xshift=2, pos=.65,left]{$p_bL^2(\mathrm{End}(a\,{\oplus}\, b))p_a$} (-\h*1.1,4);
\draw[-stealth'] (\h*1.1,4) --	node[xshift=0, pos=.38,right]{$\mathrm{Hilb}$ is a}
					node[xshift=1, pos=.56,right]{W*-cat.} (\h*1.1,1);
\draw[-stealth'] (10+\h*1.1,1) --	node[xshift=6, pos=.32,right]{$w\in T$}
					node[xshift=24, above, rotate=90]{$\mapsto$}
					node[xshift=0, pos=.68,right]{$\mathrm{End}(w)$} (10+\h*1.1,4);
\draw[-stealth'] (10-\h*1.1,4) --	node[xshift=-12, pos=.32,left]{$A$}
					node[xshift=-28, above, rotate=-90]{$\mapsto$}
					node[xshift=0, pos=.68,left]{$\mathrm{Bim}(A)$} (10-\h*1.1,1);
\draw[-stealth'] (2.3-\h*.6,1+\h*.8) --node[scale=.9, above, sloped, rotate=3, pos=.48]{\,$C\ni c\mapsto \mathrm{End}(c)$} (7.9-\h*.6,4+\h*.8);
\draw[-stealth'] (7.9+\h*.6,4-\h*.8) --node[scale=.9, below, sloped, rotate=3, pos=.49, yshift=0]{$\big(A\text{-}\mathrm{Mod},L^2A\big) \mapsfrom A$} (2.3+\h*.6,1-\h*.8);
\draw[-stealth'] (3,0+\h) --node[scale=.9, above]{\,$C\mapsto \mathrm{End}(C)$} (6.7,0+\h);
\draw[-stealth'] (6.7,0-\h) --	node[scale=1, below]{\footnotesize ideal of}
					node[scale=1, below, yshift=-10]{\footnotesize absorbing objects} (3,0-\h);
\draw[-stealth'] (0,-1) --node[scale=1.1, xshift=0, pos=.3,right]{\footnotesize the collection}
					node[xshift=0, pos=.42,right]{\footnotesize of all}
					node[xshift=1, pos=.57,right]{\footnotesize W*-categories}
					node[xshift=1, pos=.66,right]{\footnotesize is a ...}
					(0,-4);
\draw[-stealth'] (10,-1) --node[scale=1.1, xshift=0, pos=.35,right]{\footnotesize bimodules}
					node[xshift=1, pos=.49,right]{\footnotesize over a}
					node[scale=1.1, xshift=1, pos=.61,right]{\footnotesize bicom. cat. ...}
					(10,-4);
\draw[-stealth'] (7.9,-1) --	node[scale=.9, right, pos=.48, xshift=20]{Reps of a}
					node[scale=.9, right, pos=.61, xshift=20]{bicom. category ...} (2.3,-4);
}
};
\end{tikzpicture}}

\end{document}